\documentclass[10pt,centertags]{amsart}

\usepackage{amsmath}
\usepackage{amsthm}
\usepackage{amssymb}
\usepackage{mathrsfs} 
\usepackage[a4paper]{hyperref} 
\usepackage{exscale}
\usepackage{verbatim}
\usepackage[arrow, matrix, curve]{xy}
\usepackage{enumerate}
\usepackage[T1]{fontenc} 
\usepackage[utf8]{inputenc}
\usepackage[ngerman,american]{babel} 
\usepackage{amsthm}

\usepackage{graphicx}
\usepackage{color}

\usepackage{xspace}

\allowdisplaybreaks

\theoremstyle{definition}
	\newtheorem{definition}{Definition} 
	\newtheorem*{definition*}{Definition}
	
	\numberwithin{definition}{section}
\theoremstyle{plain}
	\newtheorem{lemma}[definition]{Lemma}
	\newtheorem{proposition}[definition]{Proposition}
	\newtheorem{theorem}[definition]{Theorem}
	\newtheorem*{theorem*}{Theorem}
	\newtheorem{corollary}[definition]{Corollary}

	\newtheorem*{claim*}{Claim}
\theoremstyle{remark}

	\newtheorem{remark}[definition]{Remark}

\renewcommand{\phi}{\varphi}

\newcommand{\Z}{\mathbb{Z}}

\newcommand{\R}{\mathbb{R}}

\newcommand{\xto}{\xrightarrow}

\usepackage{cancel}

\begin{document}
\bibliographystyle{plain}

\title[Quasioutomorphism groups of free groups]{On quasioutomorphism groups of free groups and their transitivity properties}
\author{Tobias Hartnick }
\address{Mathematic Department, Technion, Haifa, 32000, Israel}
\email{hartnick@tx.technion.ac.il}
\author{Pascal Schweitzer}
\address{RWTH Aachen University, Ahornstr. 55, 52074 Aachen, Germany}
\email{schweitzer@informatik.rwth-aachen.de}
\date{\today}

\begin{abstract} We introduce a notion of quasimorphism between two arbitrary
groups, generalizing the classical notion of Ulam. We then define and study the category of homogeneous quasigroups, whose objects are groups and whose morphisms are equivalence classes of quasimorphisms in our sense. We call the automorphism group ${\rm QOut(G)}$ of a group $G$ in this category the quasioutomorphism group. It acts on the space of real-valued homogeneous quasimorphisms on $G$ extending the natural action of ${\rm Out}(G)$. We discuss various classes of examples of quasioutomorphisms of free groups.
We use these examples to show that the orbit of ${\rm Hom}(F_n, \R)$ under ${\rm QOut(F_n)}$ spans a dense subspace. This is contrast to the classical fact that the corresponding ${\rm Out}(F_n)$-orbit is closed and of uncountable codimension. We also show that ${\rm QOut}(\Z^n) = {\rm GL}_n(\R)$.
\end{abstract}

\maketitle



\section{Introduction}
The space $\mathcal Q(G)$ of real-valued quasimorphisms on a group $G$ and its subspace $\mathcal H(G)$ of real-valued homogeneous quasimorphisms are classical objects of
study in analytic and geometric group theory (see e.g. \cite{BargeGhys,FujiwaraEtAl, BestvinaFujiwara, Brooks,BuMo1,BuMo2, EpsteinFujiwara,Grigorchuk,OsinHull} and the references therein). Here a function $\alpha: G \to \R$ is called a quasimorphism if $|\alpha(gh) -\alpha(g) - \alpha(h)|$ is bounded uniformly over all $g,h \in G$ and homogeneous if $\alpha(g^n) = n\cdot \alpha(g)$ for all $g \in G$, $n \in \mathbb N$. Of particular relevance for applications in rigidity theory is the quotient space
\[{\mathcal H}_0(G) := \mathcal H(G)/{\rm Hom}(G,\R),\]
which admits a cohomological interpretation in terms of bounded cohomology \cite{Gromov, Ivanov, Grigorchuk, scl}. The simplest numerical invariant derived from the study of quasimorphisms is the dimension
\[d_0(G) := \dim{\mathcal H}_0(G) \in \mathbb N_0 \cup \{\infty\}.\]
Using central extensions of lattices in products of Hermitian Lie groups, one can show that the range of the invariant $d_0$ is all of $\mathbb N_0 \cup \{\infty\}$ \cite{RemyMonod, BuMo1}. However, in almost all other examples studied in the literature the invariant $d_0$ takes values in $\{0, \infty\}$ only. For example, $d_0(G) = 0$ for all amenable groups \cite{Johnson, Gromov} but also for certain non-amenable groups such as lattices in higher rank Lie groups \cite{BuMo1}, while $d_0(G) = \infty$ for all acylindrically hyperbolic groups \cite{FujiwaraEtAl, OsinHull, Osin}. This class of groups includes in particular all (relatively) Gromov-hyperbolic groups, all non-abelian mapping class groups and all groups of outer automorphisms of free groups of finite rank $\geq 3$. For some of these classes, $d_0(G) = \infty$ was established earlier (see in particular \cite{EpsteinFujiwara, BestvinaFujiwara}). These examples indicate that if ${\mathcal H}(G)$ is non-trivial then it tends to be infinite-dimensional (and in this case it is automatically of uncountable dimension).

\medskip

This article aims at providing a conceptual explanation for the largeness of ${\mathcal H}(G)$ in the non-trivial case. To this end we point out that there is a hidden symmetry group ${\rm QOut(G)}$ acting on $\mathcal H(G)$, extending the natural ${\rm Out}(G)$-action. In order to get a better understanding of the structure of $\mathcal H(G)$ one should take these hidden symmetries into account. Thus we will consider not only the quotient ${\mathcal H}_0(G)$, but also
\[\widehat{\mathcal H}_0(G) := \mathcal H(G)/({\rm span}({\rm QOut(G)}.{\rm Hom}(G,\R))),\]
and its reduced version
\[\widehat{\mathcal H}_0^{\rm red}(G) :=  {\mathcal H}(G)/\overline{{\rm span}({\rm QOut(G)}.{\rm Hom}(G,\R))},\]
where the closure is taken with respect to the (typically non-complete) topology of pointwise convergence on  ${\mathcal H}(G)$. While $d_0(G)$ counts the number of linearly independent non-trivial homogeneous quasimorphisms, the dimensions $d(G) := \dim \widehat{\mathcal H}_0(G)$ and $d_{\rm red}(G) := \dim \widehat{\mathcal H}_0^{\rm red}(G)$ count the number of such quasimorphisms modulo hidden symmetries (not taking or taking the topology into account respectively). Note that by definition
\[0 \leq d_{\rm red}(G) \leq d(G) \leq d_0(G) \leq \infty.\]
The following main result of this article, which we will establish in Theorem \ref{ThmMain} below, shows that at least one of the two inner inequalities is strict in general. 
\begin{theorem}\label{IntroMain}
Let $G$ be a finitely generated non-abelian free group. Then $d_0(G) = \infty$, but
\[d_{\rm red}(G)=0.\]
\end{theorem}
This indicates that taking hidden symmetries into account yields to a substantially different count of quasimorphisms. Let us now explain how these hidden symmetries arise.

We follow the very general idea that categorification of a functor exhibits its hidden symmetries. More precisely, let $\mathfrak C$ be a category and $F: \mathfrak C \to \mathfrak{Set}$ be a set-valued functor which is representable in the sense that there exists an object $C_0 \in \mathfrak C$ and a natural equivalence $F \cong {\rm Hom}_{\mathfrak C}(-, C_0)$. Then ${\rm Aut}_{\mathfrak C}(C)$ acts on $F(C)$, and thereby any such representation of a functor exhibits symmetries. While the functor of (homogeneous) quasimorphisms is not representable over the category $\mathfrak{Grp}$ of groups it turns out that it is representable over a certain modification of this category, which we discuss next.

Let us first discuss the contravariant functor ${\mathcal Q}$ which to each group $G$ associates the space ${\mathcal Q}(G)$ of real-valued homogeneous quasimorphism on $G$ and to every group homomorphism $f: G \to H$ the pullback $f^*: {\mathcal Q}(H) \to {\mathcal Q}(G)$ given by $f^*\alpha := \alpha \circ f$. We would like to define a category ${\mathfrak{QGrp}}(G)$, whose objects are groups and whose morphisms are certain maps between groups such that ${\rm Hom}_{{\mathfrak{QGrp}}}(G; \R) \cong {\mathcal Q}(G)$. In any such category we need to ensure at least that for all $f \in {\rm Hom}_{{\mathfrak{QGrp}}}(G,H)$ and $\alpha \in {\mathcal Q}(H)$ the composition $f^*\alpha := \alpha \circ f$ is contained in ${\mathcal Q}(G)$. Following a suggestion by Uri Bader, our idea now is to take the universal category satisfying this property:
\begin{definition}\label{DefQM} A map $f: G \to H$ is called a \emph{quasimorphism} if for all $\alpha \in {\mathcal Q}(H)$ we have $f^*\alpha \in {\mathcal Q}(G)$. The category ${\mathfrak{QGrp}}(G)$ of \emph{quasigroups} is defined as the category whose objects are groups and whose morphisms are quasimorphisms.
\end{definition} 
Note that, by the very definition, quasimorphisms compose, whence $\mathfrak{QGrp}(G)$ is indeed a category. It is easy to see\footnote{All claims made in the introduction will be carefully established in Subsection \ref{SecBasic} below.} that indeed $ {\rm Hom}_{{\mathfrak{QGrp}}}(G,\R) =\mathcal Q(G)$. We deduce that ${\rm Aut}_{{\mathfrak{QGrp}}}(G)$ acts on $\mathcal Q(G)$, extending the action of ${\rm Aut}_{\mathfrak Grp}(G)$. 

Our definition of quasimorphism is substantially more general than previous definitions studied in the literature. It is instructive to compare our definition to the more restricted classical definition of Ulam \cite[Chapter 6]{Ulam}, which was recently (and independently from the present work) studied by Fujiwara and Kapovich \cite{FujiwaraKapovich}. They show (among many other things) that between hyperbolic groups there are no non-trivial bijective quasimorphisms in the sense of Ulam. On the contrary it follows from Theorem \ref{IntroMain} that with our more general definition there are plenty of bijective quasimorphisms between free groups  (see also the second part of Theorem \ref{SummaryQOut} below).  We thus obtain a rich theory even for hyperbolic groups. Interestingly enough, it suffices to consider a slight generalization of Ulam's condition to obtain many interesting examples. We will discuss this and various related notions of quasimorphisms in Section \ref{SecComparison} below. In the sequel we will always use the term quasimorphism in the sense of Definition \ref{DefQM}; for distinction we will refer to classical quasimorphisms $G \to \R$ as \emph{real-valued quasimorphisms}.

Similar to the functor $\mathcal Q(G)$ we can also represent the functor $\mathcal H(G)$. For this it is convenient to think of $\mathcal H(G)$ not as a subspace, but rather as a quotient of $\mathcal Q(G)$. Namely, let us call $\alpha, \beta \in \mathcal Q(G)$ \emph{equivalent} (denoted $\alpha \sim \beta$) if $\alpha-\beta$ is a bounded function. Then the inclusion $\mathcal H(G) \hookrightarrow \mathcal Q(G)$ induces a linear isomorphism
\[
\mathcal H(G) \xto{\cong}\mathcal Q(G)/\mathord\sim. 
\]
In the sequel we will always tacitly identify $\mathcal H(G)$ with this quotient.
\begin{definition} Two quasimorphisms $f, g \in  {\rm Hom}_{{\mathfrak{QGrp}}}(G,H)$ are \emph{equivalent} (denoted $f \sim g$) if for every $\alpha \in \mathcal Q(G)$ we have $f^*\alpha \sim g^*\alpha$. We write $[f]$ for the equivalence class of the quasimorphism $f$.
The category ${\mathfrak{HQGrp}}(G)$ of \emph{homogeneous quasigroups} is defined as the category whose objects are groups and whose morphisms are equivalence classes of quasimorphisms.\end{definition}
It is easy to see that composition of equivalence classes of quasimorphisms is well-defined by choosing representatives, thus ${\mathfrak{HQGrp}}(G)$ is indeed a category, and it is also easy to see that $ {\rm Hom}_{{\mathfrak{HQGrp}}}(G,\R) =\mathcal H(G)$. The category ${\mathfrak{HQGrp}}(G)$ has quite a bit more structure then ${\mathfrak{QGrp}}(G)$. Namely, it is an additive category with addition on ${\rm Hom}(G, H)$ given by $[f] \oplus [g] = [fg]$ where $fg(x) := f(x)g(x)$. The natural equivalence $ {\rm Hom}_{{\mathfrak{HQGrp}}}(G,\R) \cong \mathcal H(G)$ respects this abelian group structure. 
\begin{definition} The \emph{quasioutomorphism group} of a group $G$ is \begin{equation}{\rm QOut}(G) := {\rm Aut}_{{\mathfrak{HQGrp}}}(G).\end{equation}
\end{definition}
We will see below that for every group $G$ the composition of the canonical inclusion ${\rm Aut}_{\mathfrak Grp}(G) \to {\rm Aut}_{{\mathfrak{QGrp}}}(G)$ with the canonical projection ${\rm Aut}_{{\mathfrak{QGrp}}}(G) \to {\rm Aut}_{{\mathfrak{HQGrp}}}(G) = {\rm QOut}(G)$ factors through the outer automorphism group ${\rm Out}(G)$ of $G$. We thus obtain a natural map (in general neither injective nor surjective)
\[
{\rm Out}(G) \to {\rm QOut}(G),
\]
with respect to which the corresponding actions on $\mathcal H(G)$ are equivariant. In the case where $G = F$ is a finitely-generated free group this map is actually an inclusion.
Theorem \ref{IntroMain} implies that, unlike {\rm Out}(F), the group ${\rm QOut}(F)$ has a dense orbit in $\mathcal H(G)$. At present we have a complete understanding of the group ${\rm QOut}(G)$  only for amenable groups $G$. For finitely-generated non-abelian free groups $F$ we understand certain large subgroups, but not the full structure of ${\rm QOut}(F)$. The following result summarizes some of our main results. 
 \begin{theorem}\label{SummaryQOut}
\begin{itemize}
\item[(i)] If $G$ is amenable and its abelianization has finite rank $r$, then 
\[{\rm QOut}(G) \cong {\rm GL}_r(\R).\]
\item[(ii)] If $G$ is a finitely-generated non-abelian free group, then ${\rm QOut}(G)$ is an uncountable non-amenable group, which contains ${\rm Out}(G)$ and torsion of arbitrary order.
\end{itemize}
\end{theorem}
We will establish the various claims of the theorem in Theorem \ref{AmenableMain1}, Corollary \ref{CorNonAmenability} and Theorem \ref{thm:wobbling:embedding:and:consequences} below.

This article is organized as follows: In Section \ref{SecRQM} we first recall some background concerning real-valued quasimorphisms. We then discuss in some detail the case of free groups, and in particular the ${\rm Out}(F_n)$-action on $\mathcal H(F_n)$.The main result of that section, which is an extension of a classical theorem of Grigorchuk and might be of independent interest, constructs a dense, countable dimensional ${\rm Out}(F_n)$-invariant subspace of $\mathcal H(F_n)$. We then start discussing the categories of quasigroups and homogeneous quasigroups in Section \ref{SecHQGrp}. In particular, we show that the category of homogeneous quasigroups is equivalent to a much smaller subcategory of quasi-separated homogeneous quasigroups. In Section \ref{SecAmenable} we apply these general techniques to compute the quasioutomorphism group of an amenable group with abelianization of finite rank. In Section \ref{SecFreeQM} we study quasioutomorphism groups of free groups and establish our main theorem. Finally, in Section \ref{SecComparison} we compare our results to those of Fujiwara and Kapovich and discuss various related classes of quasimorphisms.

\section{Preliminaries on real-valued quasimorphisms}\label{SecRQM}
\subsection{General properties of $\mathcal H(G)$}
Since the definition of the category of (homogeneous) quasigroups involves real-valued quasimorphism, we need to recall some basic properties of such quasimorphisms first. Recall from the introduction that a map $\alpha: G \to \R$ is called a \emph{quasimorphism} if
\[D(\alpha) := \sup_{g,h \in G} |\alpha(gh) - \alpha(g) -\alpha(h)| < \infty.\]
In this case the real number $D(\alpha)$ is referred to as the \emph{defect} of $\alpha$. Recall also that a quasimorphism is \emph{homogeneous} if $\alpha(g^n) = n\cdot\alpha(g)$ for all~$g \in G$ and $n \in \mathbb N$, and that two quasimorphisms are called \emph{equivalent} if their difference is a bounded function. 
Finally we remind the reader of our notations $\mathcal Q(G)$ and $\mathcal H(G)$ for the spaces of all, respectively all homogeneous real-valued quasimorphisms on $G$. We will usually denote real-valued quasimorphisms by small Greek letters. The following facts are elementary, see e.g. \cite[Sec. 2.2]{scl} for proofs.
\begin{lemma}\label{sclStuff}
\begin{itemize}
\item[(i)] If $p: G \to H$ is a group homomorphism and $\alpha: H \to \R$ a (homogeneous) quasimorphism, then $p^*\alpha := \alpha \circ p$ is a (homogeneous) quasimorphism.
\item[(ii)] Every quasimorphism $\alpha$ is equivalent to a unique homogeneous quasimorphism $\widehat{\alpha}$, called its homogenization. In particular, every bounded homogeneous quasimorphism is trivial.
\item[(iii)] Every homogeneous quasimorphism is conjugation-invariant and satisfies $\alpha(g^{-1}) = -\alpha(g)$.
\item[(iv)] Every homogeneous quasimorphism on an amenable group is a homomorphism.
\end{itemize}
\end{lemma}
This shows in particular that $\mathcal H(G) \cong \mathcal Q(G)/\sim$. Moreover, by (i)  we have an action of ${\rm Aut}_{\mathfrak{Grp}}(G)$ on $\mathcal H(G)$, which by (iii) factors through an action of ${\rm Out}(G)$. The ${\rm Out}(G)$-module $\mathcal H(G)$ admits the following cohomological interpretation: Denote by $H^2(G; \R)$ and $H^2_b(G; \R)$ the second group cohomology, respectively second bounded group cohomology of $G$ with trivial coefficient module $\R$ \cite{Ivanov, MonodDiss}. Then there is a natural comparison map $H^2_b(G; \R) \to H^2(G; \R)$ whose kernel $EH^2_b(G; \R)$ satisfies
\begin{equation}\label{CohomologicalInterpretation}
EH^2_b(G; \R) \cong {\mathcal H}_0(G) := \mathcal H(G)/{\rm Hom}(G,\R).\end{equation}
Moreover, this isomorphism is compatible with the respective ${\rm Out}(G)$-actions \cite{Grigorchuk, MonodDiss}. The space $H^2_b(G; \R)$ can be equipped with the structure of a Banach space~\cite{MonodDiss}. It follows that for finitely generated groups $G$ or more generally for all groups $G$ with countable-dimensional $H^2(G; \R)$ and countable-dimensional space ${\rm Hom}(G,\R)$, the space $\mathcal H(G)$ is either finite-dimensional or uncountable-dimensional. To deal with the latter case it is convenient to introduce a topology on $\mathcal H(G)$. On $\mathcal Q(G)$ we can consider the (locally convex) topology of pointwise convergence, i.e., $\alpha_n \to \alpha$ provided $\alpha_n(g) \to \alpha(g)$ for all $g \in G$. As discussed above we can consider $\mathcal H(G)$ either as a subspace or as a quotient of $\mathcal Q(G)$ and equip it with either the quotient topology or the subspace topology. These two topologies are \emph{not} the same. Consider e.g. the case $G = \Z$ and define functions 
\[
b_n(k) := \left\{\begin{array}{ll} k  & |k|< n,\\ n & |k| \geq n.\end{array} \right.
\]
Then $b_n \in \mathcal Q(\Z)$ and $b_n \to {\rm Id}_{\Z}$. In the quotient ${\mathcal Q}(\Z)/\mathord\sim$ we have $[b_n] = 0$, since $b_n$ is bounded. Thus with respect to the quotient topology we have $[{\rm Id}_\Z] \in \overline{\{0\}}$. In particular, ${\mathcal Q}(\Z)/\mathord\sim$ is not Hausdorff, whereas $\mathcal H(G) \subseteq \mathcal Q(G)$ is Hausdorff for any group $G$ when equipped with the subspace topology. For this reason, we will always equip $\mathcal H(G)$ with the subspace topology in the sequel.

There are two main advantages of this topology. Firstly, the action of ${\rm Out}(G)$ is by homeomorphisms. Indeed, if $\alpha_n \to \alpha$ pointwise, then for all $f \in {\rm Out}(G)$ and $g \in G$ we have $f^*\alpha_n(g) = \alpha_n(f(g)) \to \alpha(f(g)) = f^*\alpha(g)$. Secondly, the topology has good separability properties. For example, assume that $G$ is a finitely generated non-abelian free group. Then $\mathcal H(G)$ admits a countable dimensional dense subspace with respect to the topology of pointwise convergence \cite{Grigorchuk}, whereas the norm topology on $H^2_b(G; \R)$ is not separable (see \cite[Cor. 18]{Rolli}). The main disadvantage of our topology is that it is not complete since the pointwise limit of homogeneous quasimorphisms is not necessarily a homogeneous quasimorphism. It is unclear to us whether an ${\rm Out}(G)$-invariant complete separable locally convex topology on $\mathcal H(G)$ exists.

\subsection{The case of free groups I: Counting quasimorphisms} In the remainder of this section we
are going to discuss the case of a finitely generated non-abelian free group $F$ in some detail. This serves the double purpose to illustrate the notions introduced in the previous subsection and to prepare our study of ${\rm QOut}(F)$ below. In view of the latter goal, we will introduce some concepts in greater generality than needed here. Our main goal is to exhibit an explicit subspace of $\mathcal H^*(F)<\mathcal H(F)$ which is at the same time countable-dimensional, ${\rm Out}(F)$-invariant and dense (with respect to the topology introduced in the previous subsection). Since ${\rm Out}(F)$ acts continuously on $\mathcal H(F)$ and $\mathcal H^*(F)$ is dense, we can deduce that the ${\rm Out}(F)$-action on $\mathcal H(F)$ is uniquely determined by its restriction to $\mathcal H^*(F)$, so that we can think of elements of  ${\rm Out}(F)$ as countably infinite matrices.

A countable-dimensional dense subspace of $\mathcal H(F)$ was first constructed by Grigorchuk in \cite{Grigorchuk}. This subspace is however not ${\rm Out}(F)$-invariant, so we will have to enlarge it to an ${\rm Out}(F)$-invariant space. Both Grigorchuk's construction and ours are based on the notion of a counting quasimorphism. These quasimorphisms and their generalizations play a major role in the modern theory of quasimorphisms, cf. \cite{FujiwaraEtAl, OsinHull, EpsteinFujiwara, BestvinaFujiwara}. In the case of a finitely-generated free group $F$ they can be defined as follows: Let $S$ be a free generating set of $F$ and identify $F$ with the set of reduced words over $S \cup S^{-1}$, including the empty word $\varepsilon$. Given two words $w_1$, $w_2$ over  $S\cup S^{-1}$ we write $w_1 = w_2$ if they coincide and $w_1 \equiv w_2$ if they define the same element of $F$. Of course, for reduced words these notions coincide. Given two reduced words $w = y_1\cdots y_l$ and $w_0 = x_1\cdots x_n$ over $S\cup S^{-1}$, a \emph{$w_0$-subword} of $w$ is a sequence $y_j\cdots y_{j+n-1}$ with $y_l = x_{l-j+1}$ for all $l \in \{j, \dots, j+n-1\}$. Two $w_0$-subwords $y_j\cdots y_{j+n-1}$ and $y_k\cdots y_{k+n-1}$ are said to \emph{overlap} if $\{j, \dots, j+n-1\} \cap \{k, \dots, k+n-1\}\neq \emptyset$. A family of $w_0$-subwords of $w$ is called \emph{non-overlapping} if 
they do not overlap pairwise. We denote by $\#_{w_0}(w)$ the maximal number of distinct (but potentially overlapping) $w_0$-subwords of $w$ and by $\#^*_{w_0}(w) \leq \#_{w_0}(w)$ the maximal number of non-overlapping $w_0$-subwords of $w$. Thus for instance $\#_{ss}(sssss) = 4$ and $\#^*_{ss}(sssss) =2$.

 \begin{definition}
For a reduced word~$w_0$ over $S \cup S^{-1}$ we define define maps $\phi_{w_0}: F \to \Z$ and $\phi^*_{w_0}: F \to \Z$ as follows. For every element~$g \in F$ let $w_g$ be the unique reduced word over $S \cup S^{-1}$ with $g \equiv w_g$. Then we set \[\phi_{w_0}(g) := \#_{w_0}(w_g) - \#_{w_0^{-1}}(w_g)  \quad \text{ and } \quad
\phi^*_{w_0}(g) := \#^*_{w_0}(w_g) - \#^*_{w_0^{-1}}(w_g).\]
\end{definition}

It is easy to check that both $\phi_{w_0}$ and $\phi^*_{w_0}$ are quasimorphisms, called the \emph{overlapping counting quasimorphism} and the \emph{non-overlapping counting quasimorphism} associated with $w_0$ respectively. In \cite{scl} these are called the \emph{big counting quasimorphism} and the \emph{little counting quasimorphism} respectively. We will almost exclusively work with the $\phi_{w_0}$ and thus simply call them counting quasimorphism for short.
 Note that all notions discussed so far depend crucially on the choice of the free generating set $S$.
 
 In general, counting quasimorpisms are not homogeneous. We denote by $\widehat{\phi_{w_0}}$ the homogenization of $\phi_{w_0}$ and refer to such quasimorphisms as \emph{homogenized overlapping counting quasimorphisms}, or \emph{hoc-quasimorphisms} for short. We are going to prove the following equivariant version of Grigorchuk's theorem:
\begin{theorem}\label{HStar}
Let $\mathcal H^*(F,S)$ denote the subspace of $\mathcal H(F)$ spanned by the hoc quasimorphisms with respect to the free generating set $S$. Then
\begin{itemize}
\item[(i)] $\mathcal H^*(F,S)$ is invariant under the action of ${\rm Out}(F)$, hence independent of the generating set $S$.
\item[(ii)] $\mathcal H^*(F) := \mathcal H^*(F,S)$ is dense in $\mathcal H(F)$ with respect to the topology of pointwise convergence.
\item[(iii)] $\mathcal H^*(F)$ is of countable dimension.
\end{itemize}
\end{theorem}
The remainder of this section is devoted to the proof of Theorem \ref{HStar}. We will complete the proof in Subsection \ref{SecProofHStar} below. 

\subsection{The case of free groups II: Grigorchuk's theorem} In this section we recall Grigorchuk's construction of a dense, countable-dimensional subspace of $\mathcal H(F)$, which immediately implies Parts (ii) and (iii) of Theorem \ref{HStar}. Throughout we fix a free generating set $S$ of $F$.
Since we will use a variation of the ideas behind Grigorchuk's proof in our study of word-exchange quasimorphisms in Section \ref{SecWordExchange} below, we introduce the relevant concepts in slightly greater generality than necessary for the goal at hand.
Let $x_1, \dots, x_n, y_1, \dots, y_m \in S \cup S^{-1}$ and consider two reduced words $w_1 = x_1\cdots x_n$ and~$w_2 = y_1\cdots y_m$. We say that some proper postfix of~$w_1$ is a proper prefix of~$w_2$ if there is an integer~$i\in\{1,\ldots,n-1\}$ such that~$n-i < m$ and for all~$t\in \{1,\ldots, n-i\}$ we have~$x_{t+i} = y_t$. We say~$w_2$ is a subword of~$w_1$ if there is an integer~$i\in\{0,\ldots,n-1\}$ such that~$n-i \geq m$ and for all~$t\in \{1,\ldots, n-i\}$ we have~$x_{t+i} = y_t$. We say~$w_1$ and~$w_2$ \emph{overlap} if a proper prefix of one of the words is a proper postfix of the other word or one of the words is a proper subword of the other. We write $w_1 \pitchfork w_2$ if $w_1$ and $w_2$ do not overlap and call $w$ \emph{non-self-overlapping} if $w \pitchfork w$. 
\begin{definition}\label{DefIndependent}
A set $\{w_1,\ldots,w_k\}$ of $k$ distinct, reduced, non-empty words is \emph{independent} if the following hold:
\begin{itemize}
\item[(i)] The set $\{w_1,\ldots,w_k,{w_1}^{-1},\ldots,{w_k}^{-1}\}$ has cardinality $2k$.
\item[(ii)] For any $u, u' \in \{w_1,\ldots,w_k,{w_1}^{-1},\ldots,{w_k}^{-1}\}$ (potentially equal) we have $u \pitchfork u'$.
\end{itemize}
In this case we say that the words $w_1, \dots, w_k$ are \emph{mutually independent}.
\end{definition}
General sets of independent words will be studied in Section \ref{SecWordExchange} below. 
In the context of Grigorchuk's theorem and its equivariant version we only need to consider 
\emph{self-independent words}, i.e., reduced words $w$ such that the singleton $\{w\}$ is independent. Note that, by definition, $w$ is self-independent if and only if $w^{-1}$ is self-independent. 
\begin{lemma}\label{SelfIndependenceLemma}
Assume that a word $w$ over $S$ is cyclically reduced and non-empty. Then $w^{-1}$ is also cyclically reduced and non-empty. In this case, $w$ and~$w^{-1}$ are distinct and do not overlap. 
\end{lemma}
\begin{proof} The first statement is obvious. Now assume that $w$ (and hence also $w^{-1}$) is cyclically reduced and non-empty. Assume for contradiction that $w$ and $w^{-1}$ overlap or coincide. Then after possibly exchanging $w$ and $w^{-1}$ we may assume that there exists $t \in \{1, \dots, n\}$ such that $w = x_1 \cdots x_n$, $w^{-1} = y_1 \cdots y_n$ and
\[y_j = x_{n-j+1}^{-1} \quad (j=1, \dots, n), \quad y_l = x_{n-t+l} \quad (l=1, \dots, t). \]
We distinguish two cases: If $t=2r$ is even then we get $x_{n-r+1}^{
-1} = y_r = x_{n-r}$ contradicting the fact that $x_1 \cdots x_n$ is reduced. If $t=2r+1$, then $x_{n-r}^{-1}=y_{r+1} = x_{n-r}$, which is impossible.
\end{proof}
\begin{corollary}\label{cor:SelfIndependence} Let $w$ be a non-empty reduced word over $S$.
\begin{itemize}
\item[(i)] If $w$ (or equivalently $w^{-1}$) is cyclically reduced and non-self-overlapping, then $w$ (and hence $w^{-1}$) is self-independent.
\item[(ii)] If $w$ is self-independent then $\phi_w = \phi_{w}^*$. 
\end{itemize}
\end{corollary}
\begin{proof} (i) is immediate from Lemma \ref{SelfIndependenceLemma} and (ii) is obvious .
\end{proof}
We also record the following consequence for later reference:
\begin{corollary}\label{phistarww}
If $w$ is cyclically reduced and non-trivial, then $\widehat{\phi^*_w}(w) = 1$.
\end{corollary}
\begin{proof} Since $w$ is cyclically reduced we have $\#_w(w^n) = n$. By Lemma \ref{SelfIndependenceLemma} $w$ does not overlap with $w^{-1}$ we have $\#_{w^{-1}}(w^n) = 0$. We deduce that $\phi^*_w(w^n) = n$, from which the corollary follows.
\end{proof}
From now on we denote by $\mathcal G(S)$ the set of cyclically reduced, non-self-overlapping reduced words over $S \cup S^{-1}$. We pick a total order $\leq$ on $S \cup S^{-1}$ and denote by $\preceq$ the induced (total) lexicographic order on reduced words. Note that if two words in $\mathcal G(S)$ are conjugate, then they are cyclic permutations of each other. Among these conjugates there is a unique minimal element with respect to $\preceq$.
We refer to such an element as \emph{conjugacy-minimal}. Given $w \in \mathcal G(S)$ let $w^*$ denote the conjugacy minimal element of the conjugacy class $[w]$ and let $w^{-*}$ denote the conjugacy-minimal element of $[w^{-1}]$. Then we define
\[w^\dagger := \min\{w^*, w^{-*}\}.\]
and denote by $\mathcal G(S, \leq)^+$ the collection of all the $w^\dagger$ for $w  \in \mathcal G(S)$. By construction we have $\mathcal G^+ \cap (\mathcal G^{+})^{-1} = \emptyset$, and moreover $\mathcal G^+ \cup (\mathcal G^{+})^{-1} $ meets every conjugacy class in $\mathcal G(S)$ exactly once.
\begin{theorem}[Grigorchuk \cite{Grigorchuk}]\label{Grig} For any choice of order $\leq$ on $S$ the family $(\widehat{\phi_g}\,|\,g \in \mathcal G(S, \leq)^+)$ is linearly independent and its span is dense in $\mathcal {H}(G)$ with respect to the topology of pointwise convergence.
\end{theorem}
In his original proof Grigorchuk actually works with homogenizations $\widehat{\phi^*_g}$ of non-overlapping counting quasimorphisms, but because of Corollary \ref{cor:SelfIndependence} this does not make any difference. 
\subsection{The case of free groups III: The ${\rm Out}(F)$-action}\label{SecNielsen}\label{SecProofHStar}
We are now going to study the effect of the ${\rm Out}(F)$-action on hoc-quasimorphisms. Our first observation is that the ${\rm Out}(F)$-action on $\mathcal Q(F)$ preserves bounded functions, hence it descends to an action on the quotient $\mathcal Q(F)/\mathord\sim$, and under the identification $\mathcal Q(F)/\mathord\sim \cong \mathcal H(F)$ this action coincides with the usual ${\rm Out}(F)$-action on  $\mathcal H(F)$. Thus in order to understand the effect of the ${\rm Out}(F)$-action on hoc-quasimorphisms it suffices to understand the effect of the action on (non-homogeneous, overlapping) counting quasimorphisms, which are more convenient for computations.

Assume $F= F_n$ is free on $n$ generators and enumerate the generating system, writing $S = \{a_1, \dots, a_n\}$. Then ${\rm Out}(F_n)$ is generated as a group by the \emph{Nielsen transformations} $P_1, P_2, I, T$ which are respectively determined by having the following effect on the standard basis:
\begin{eqnarray*}
P_1(a_1, \dots, a_n) &=& (a_2, a_1, \dots, a_n),\\
P_2(a_1, \dots, a_n) &=& (a_2, \dots, a_n, a_1),\\
 I(a_1, \dots, a_n) &=& (a_1^{-1}, a_2\dots, a_n),\\
T(a_1, \dots, a_n) &=& (a_1a_2, a_2, \dots, a_n).
\end{eqnarray*}
The first three transformations are of finite order, and $T$ is conjugate to its inverse by a product of these. It follows that the Nielsen transformations generate ${\rm Out}(F_n)$ as a \emph{semigroup}. The first three transformations preserve reduced words. It follows immediately that for any counting quasimorphism $\phi_{w_0}$ we have
\[P_j^*\phi_{w_0}(w) = \phi_{P_j^{-1}w_0}(w), \quad I^*\phi_{w_0}(w) = \phi_{I^{-1}w_0}(w),\]
and similarly for $\phi^*_{w_0}$ instead of $\phi_{w_0}$. In particular, $\mathcal H^*(F_n, S)$ is invariant under these transformations, and it remains only to understand the action of $T$.

For~$T$, the letters~$a_1$ and~$a_2$ play a distinguished role.
To save indices we define~$a:= a_1$ and~$b := a_2$. By definition of~$T$ we then have~$T(s) = T^{-1}(s)= s$ for~$s \in S\setminus \{ a\}$. We also have 
\[T(a) = ab, \; T(a^{-1}) = b^{-1}a^{-1}, \; T^{-1}(a) = ab^{-1},\;\text{and } T^{-1}(a^{-1}) = ba^{-1}.\]

For computations involving $T$ and its inverse it is therefore convenient to use the following normal form of elements of $F_n$. Every $g \in F_n$ is uniquely represented by a reduced word of the form
\begin{equation}\label{eq-normalform-NielsenT}
g \equiv w = b^{n_0}s_1b^{n_1}\cdots s_lb^{n_l}
\end{equation}
subject to the conditions
\begin{equation}\label{eq-normalform-NielsenT-conditions}
l \geq 0, \quad n_j \in \Z, \quad s_j \in (S\cup S^{-1})\setminus\{b,b^{-1}\}, \quad \forall 1 \leq j\leq l-1: \;n_j = 0 \Rightarrow s_j \neq s_{j+1}^{-1}.\end{equation}
We will express Equation~\eqref{eq-normalform-NielsenT} by writing
\[g \equiv_S (n_0, s_1, n_1, \dots, s_l, n_l).\]
Then we have
\[Tg \equiv_S (\widetilde{n_0}, s_1, \widetilde{n_1}, \dots, s_l, \widetilde{n_l}),\]
where 
(with the convention $s_0= s_{l+1} = \varepsilon$) 
\[\widetilde{n_j} = n_j + \#_a(s_j) - \#_{a^{-1}}(s_{j+1}).\]
Note in particular, that Conditions \eqref{eq-normalform-NielsenT-conditions} are preserved. Similarly we have
\[T^{-1}g \equiv_S (\widetilde{n_0}^*, s_1, \widetilde{n_1}^*, \dots, s_l, \widetilde{n_l}^*),\]
where
\[\widetilde{n_j}^* = n_j - \#_a(s_j) + \#_{a^{-1}}(s_{j+1}).\]
To describe the action of $T$ on (certain) counting quasimorphisms, we define a \emph{truncation operator} $\tau_b$ on reduced words as follows: If $w$ is given as in Equation~\eqref{eq-normalform-NielsenT}, then
\[\tau_b(w) := s_1b^{n_1}\cdots b^{l-1} s_l,\]
i.e., $\tau_b$ forces reduced words to start and end in letters different from $b$ by truncating leading and final powers of $b$. We say a word~$w$ is~\emph{$b$-truncated} if it neither ends nor starts in a~$b$ power, i.e., if~$\tau_b(w) = w$. Then $b$-truncated subwords of~$g$ bijectively correspond to~$b$-truncated subwords of~$Tg$, and this is the key to the following lemma:
\begin{lemma}\label{lem:non:b:power}
Let $w$ be a reduced word and assume that $w$ is not a power of $b$. Then $T^*\phi_w$ is at bounded distance from a finite sum of counting quasimorphism. In particular, $T^*\widehat{\phi_w} \in \mathcal H^*(F_n, S)$.
\end{lemma}
\begin{proof}  Using our normal form we can write $w \equiv_S (m_0, r_1, m_1, \dots, r_k, m_k)$.
Let~$g\equiv_S (n_0, s_1, n_1, \dots, s_l, n_l)$ be an arbitrary word in~$F$. Our goal is to determine a formula for~$\#_w(Tg)$. As we noted above there is a one to one correspondence of~$b$-truncated subwords of~$T(g)$ to~$b$-truncated subwords of~$g$.
If~$w' = (0, s_{t+1}, n_{t+1}, \dots, s_{t+k}, 0)$ is a~$b$-truncated subword of~$Tg$ then its corresponding~$b$-truncated subword in~$g$ is~$\widetilde{w'}^{*} = (0,s_{t+1}, \widetilde{n_{t+1}}^*, \dots, s_{t+k},0)$ and vice versa. In particular, the number of occurrences of $w'$ in $Tg$ is the same as the number of occurrences of $\widetilde{w'}^{*}$ in $g$. 

In order to count the occurrences of $w$ in $Tg$ we proceed as follows: Every $w$-subword of $Tg$ leads to a subword $w'$ as above satisfying 
\begin{equation}\label{w'}
\forall i\in\{1,\ldots,k-1\}: r_i = s_{i+t}, \quad m_i = n_{i+t}.
\end{equation}
The corresponding subword  $\widetilde{w'}^{*}$ in $g$ then satisfies $\widetilde{m_{i}}^{*} = \widetilde{n_{i+t}}^{*}$. Let us call $w'$ satisfying \eqref{w'} \emph{extendible} if it arises from an occurrence of $w$. This means that $w'$ is preceded by at least $|m_0|$ copies of $b$ or $b^{-1}$ (according to whether $m_0 \geq 0$ or $m_0 < 0$) and proceeded by at least $|m_k|$ copies of $b$ or $b^{-1}$ (according to whether $m_k \geq 0$ or $m_k < 0$). We are now going to reformulate these conditions in terms of $g$. To exclude border cases we suppose~$t>0$ and~$t+k< l$. The remaining cases will only result in an additive error of at most 2 in our count of $\#_w(T(g))$, which is irrelevant. Our assumption ensures in particular that~$s_{t}$ and~$s_{t+k+1}$ are well defined. For $w'$ to be extendible the following conditions must hold: For the left end of the word, it is necessary that~$m_0 \cdot  n_t \geq 0$ and~$|n_t|\geq |m_0|$. Similarly on the right side of the word, it is necessary that~$m_{k} \cdot n_{t+k} \geq 0$ and~$|n_{t+k}|\geq |m_k|$.

We translate these conditions on~$w'$ into conditions on~$\widetilde{w'}^{*}$. For the left end of the word, it requires that one of the following holds:
\begin{center}$\begin{array}{lllcccl}
m_0 = 0,\\
m_0 > 0    &\wedge &\widetilde{n_{t}}^{*} \geq m_0+1,\\
m_0 > 0    &\wedge &\widetilde{n_{t}}^{*} = m_0       &\wedge &(s_{t} = a \vee   r_{1} \neq a^{-1}),\\
m_0 > 0    &\wedge &\widetilde{n_{t}}^{*} = m_0-1     &\wedge & s_{t} = a \wedge r_{1} \neq a^{-1},\\
m_0 < 0    &\wedge &\widetilde{n_{t}}^{*} \leq m_0-1,\\
m_0 < 0    &\wedge &\widetilde{n_{t}}^{*} = m_0       &\wedge & (s_{t} \neq a \vee r_{1} = a^{-1}), &\text{ or}\\
m_0 < 0 &\wedge &\widetilde{n_{t}}^{*} = m_0+1     &\wedge & s_{t} \neq a \wedge r_{1} = a^{-1}.
\end{array}$\end{center}

Note that these cases are disjoint. Since~$m_0$ and~$r_1$ are determined by the fixed word~$w$, the conditions only depend on~$s_t$ and~$\widetilde{n_{t}}^{*}$. Moreover, for the cases involving an inequality for~$\widetilde{n_{t}}^{*}$ the value of~$s_{t}$ is irrelevant. Thus, whether the conditions are fulfilled depends only on at most~$|m_0|+1$ letters to the left of~$\widetilde{w'}^{*}$.
The requirements on the right side are similar by the symmetry of taking inverses and the same conclusions hold. This implies that the conditions can be expressed by a finite set of words~$\mathcal{W}$ in which~$\widetilde{w'}^{*}$ must be contained. 
More precisely: Let~$\mathcal{W}_{\text{left}} = $
 \[\begin{cases}
\{\varepsilon\}& 
\text{ if~$m_0 = 0$},\\
\{b^{m_0}\} \cup \{ab^{m_0-1}\} & \text{ if~$m_0 > 0$ and~$r_{t+1} \neq a^{-1}$},\\
\{b^{m_0+1}\}\cup \{ab^{m_0}\} & \text{ if~$m_0 > 0$ and~$r_{t+1} = a^{-1}$},\\
\{sb^{m_0} \mid s \in (S\cup S^{-1})\setminus \{a,b,b^{-1}\}\} \cup \{b^{m_0-1}\} & \text{ if~$m_0 < 0$ and~$r_{t+1} \neq a^{-1}$},\\
\{sb^{m_0+1} \mid s \in (S\cup S^{-1})\setminus \{a,b,b^{-1}\}\} \cup \{b^{m_0}\} & \text{ if~$m_0 < 0$ and~$r_{t+1} = a^{-1}$}.\\
\end{cases}\]

Similarly we can define~$\mathcal{W}_{\text{right}}$.
Consider the set of words~$\mathcal{W} = \mathcal{W}_{\text{left}} \cdot \widetilde{w'}^{*} \cdot  \mathcal{W}_{\text{right}}$. 
If~$w'$ is a~$b$-truncated word of~$T(g)$ that is not a border case and~$\widetilde{w'}^{*}$ is the corresponding~$b$-truncated word of~$g$ then~$w'$ is extendible if and only if~$\widetilde{w'}^{*}$ is contained in a word in~$\mathcal{W}$ in~$g$.

By definition, the number of occurrences of $w$ in $T(g)$ is precisely the number of extendible $w'$s or, equivalently, the corresponding $\widetilde{w'}^*$s. We claim that this number is precisely given by the number of subwords that are equal to some word in $\mathcal W$, in other words, every $\widetilde{w'}^*$ can only be  extended in one way to an element in $\mathcal W$. Let us call a word in~$\mathcal{W}$ an extension of a word equal to~$\widetilde{w'}^{*}$ if it is obtained by attaching a word from~$\mathcal{W}_{\text{left}}$ to the left and a word from~$\mathcal{W}_{\text{right}}$ and to right.
We need to argue that for every word equal to~$\widetilde{w'}^{*}$ there is at most one extension to a word in~$\mathcal{W}$ and every word in~$\mathcal{W}$ is the extension of at most one word equal to~$\widetilde{w'}^{*}$.
Another way of saying this is the following claim: If~$w_l  {w'}^{*} w_r = w'_l  {w'}^{*} w'\maketitle
_r$ with~$w_l,w'_l \in \mathcal{W}_{\text{left}}$ and~$w_r,w'_r \in \mathcal{W}_{\text{right}}$ then~$w_l = w_l'$ and~$w_r = w'_r$. This claim follows from the fact that no word in~$\mathcal{W}_{\text{left}}$ can be extended to another word in~$\mathcal{W}_{\text{left}}$ by adding letters on the right and the symmetric fact for~$\mathcal{W}_{\text{right}}$. 

We have thus shown that~$| \#_w(T(g)) - \sum_{u \in \mathcal{W}} \#_u(g) | \leq 2$. This implies that~$T^*\phi_w$ is equivalent to~$\sum_{u\in \mathcal{W}} \phi_u$.
\end{proof}
In order to obtain $T$-invariance of $H^*(F_n, S)$ it remains to deal with counting quasimorphisms of the form $\phi_{b^k}$. Here we use the following reduction step:
\begin{lemma}\label{ReductionTrickGrig} Let $s \in S$ and let $w$ a reduced word with last letter $s$. Then the quasimorphism
\[\phi_w - \sum_{s' \in S\setminus\{s^{-1}\}} \phi_{ws'}\]
is bounded, hence has trivial homogenization.
\end{lemma}
\begin{proof} Assume $w$ occurs as a subword in some reduced word $w_0$. Then either $w$ contains the last letter, or there exists a letter $s'$ after the occurrence, which is different from $s$. Thus the above difference counts the difference between then number of occurrences as $w$ containing the last letter and the number of occurrences of $w^{-1}$ containing the first letter, which is bounded in absolute value by $1$.
\end{proof}
A variant of this lemma for certain non-overlapping counting quasimorphisms appears already in \cite{Grigorchuk}. However, for the lemma to hold in general, we need to work with overlapping rather than non-overlapping counting quasimorphisms. While the lemma is useful for the understanding of the ${\rm Out}(F_n)$-module structure of $\mathcal H(F_n)$ in general, here we are only interested in the following simple consequence:
\begin{lemma}\label{LemmabPowers} For~$k\in \mathbb{Z}$,  $T^*\phi_{b^k} \in \mathcal H^*(F_n, S)$.
\end{lemma}
\begin{proof}
If suffices to show the statement for~$k> 0$.
For~$k=1$ we have~$T^*(\phi_b) = \phi_a+\phi_b$. For~$k>1$, by Lemma~\ref{ReductionTrickGrig}, $\phi_{b^k}$ is equivalent to~$\phi_{b^{k-1}} - \sum_{s\in (S\cup S^{-1})\setminus \{b,b^{-1}\} }\phi_{b^{k-1}s}$ and the lemma follows by  induction on~$k$ using  Lemma~\ref{lem:non:b:power}.
\end{proof}
We deduce:
\begin{corollary}\label{InvarianceHStar} The space $\mathcal H^*(F_n, S)$ is invariant under ${\rm Out}(F_n)$.
\end{corollary}
\begin{proof} Combining Lemma \ref{lem:non:b:power} and Lemma \ref{LemmabPowers} we see that  $\mathcal H^*(F_n, S)$ is invariant under $T$. Since we have seen above that $\mathcal H^*(F_n, S)$ is also invariant under the other Nielsen transformations and since these generate ${\rm Out}(F_n)$ as a semigroup, we deduce that  $\mathcal H^*(F_n, S)$ is invariant under ${\rm Out}(F_n)$. 
\end{proof}
We can now finish the proof of  Theorem \ref{HStar}.
\begin{proof}[Proof of Theorem \ref{HStar}]  We have seen in Corollary \ref{InvarianceHStar} that $\mathcal H^*(F, S)$ is invariant under ${\rm Out}(F)$. Since ${\rm Out}(F)$ acts transitively on free generating sets, it follows that $\mathcal H^*(F, S)$ is independent of the choice of $S$. This shows (i). Moreover, $\mathcal H^*(F, S)$ contains the span of the family $(\widehat{\phi_g}\,|\,g \in \mathcal G(S, \leq)^+)$. By Theorem \ref{Grig}, this span is countable dimensional and
dense in $\mathcal H(F)$, hence $\mathcal H^*(F, S)$ is also dense and at least countable-dimensional. On the other hand, it is generated by countable-many elements, so it is indeed countable-dimensional. This establishes (ii) and (iii) and finishes the proof of Theorem \ref{HStar}.
\end{proof}
\begin{remark} The proof of Theorem \ref{HStar} does not exhibit any countable basis of $\mathcal H^*(F)$. It shows, however, that for any choice of ordered free generating set $(S, \leq)$ there exists a basis of $\mathcal H^*(F)$ containing the countable linearly independent subset $\{\widehat{\phi_g}\,| g\in \mathcal G(S, \leq)^+\}$. Finding such a basis would be of great interest for the understanding of the ${\rm Out}(F)$-action on $\mathcal H^*(F)$.
\end{remark}

\section{The category of homogeneous quasigroups}\label{SecHQGrp}
\subsection{Basic properties}\label{SecBasic}
We now return to the categories $\mathfrak{QGrp}$ and $\mathfrak{HQGrp}$ defined in the introduction. Our first goal is to establish some of their basic properties and in particularly to verify the various claims made in the introduction.

By definition, a map $f: G \to H$ is called a quasimorphism if $f^*\alpha$ is a quasimorphism for every real-valued quasimorphism $\alpha: H \to \R$. We claim that it suffices to check this property for all homogeneous quasimorphisms $\alpha \in \mathcal H(H)$. Indeed, by Lemma \ref{sclStuff} every real-valued quasimorphism $\alpha$ can be written uniquely as $\alpha = \widehat{\alpha} + b$ where $\widehat{\alpha}$ is a homogeneous quasimorphism and $b: H \to \R$ a bounded function. Then $f^*\alpha = f^*\widehat{\alpha} + f^*b$ and the claim follows from the fact that $f^*b$ is bounded. The same argument also shows that quasimorphisms $f_1, f_2: G \to H$ are equivalent if and only if $f_1^*\alpha-f_2^*\alpha$ is bounded for all $\alpha \in \mathcal H(H)$. This in turn is equivalent to the vanishing of the homogenization of $f_1^*\alpha-f_2^*\alpha$, or equivalently to the condition
 \begin{equation}\label{InjectivityHomMap}
 \forall \alpha \in \mathcal H(H): \widehat{f_1^*\alpha} = \widehat{f_2^*\alpha}.\end{equation}
 The following criterion is often useful in constructing quasimorphisms: \begin{lemma}\label{lem:qmorphism:criterion}
Let~$f \colon G \rightarrow H$ be a map. If there is a finite set~$E\subseteq H$ such that
$\forall g_1,g_2\in G \, \exists h \in H \colon \, f(g_1g_2) \in E f(g_1) Eh Eh^{-1}E  f(g_2)  E$, 
then~$f$ is a quasimorphism.
\end{lemma}
\begin{proof} Let $\alpha \in \mathcal H(H)$ and $g_1, g_2 \in G$. By assumption there is an $h \in H$ such that $\alpha(f(g_1 g_2)) \in \alpha( E f(g_1) Eh Eh^{-1}E  f(g_2)  E)$. Thus,
\begin{eqnarray*}
D(f^*\alpha)&=&\sup_{g_1, g_2 \in G} \left| f^*\alpha(g_1g_2) - f^*\alpha(g_1) - f^*\alpha(g_2)\right|\\
&\leq& \sup_{g_1, g_2 \in G} \sup_{e_j \in E}  \left|\alpha( e_1 f(g_1) e_2h e_3h^{-1}e_4  f(g_2)  e_5) - \alpha(f(g_1)) - \alpha(f(g_2))\right|\\
&\leq&\left|\alpha(f(g_1)f(g_2)) - \alpha(f(g_1)) - \alpha(f(g_2))\right| + 9D(\alpha) + 5\max_{e \in E} \alpha(e)\\
&\leq& 10 D(\alpha) + 5\max_{e \in E} \alpha(e) < \infty.
\end{eqnarray*}
\end{proof}
We will apply a version of this lemma to provide plenty of explicit examples of quasimorphism between free groups in Section \ref{SecFreeQM} below. We also record the following special case:
\begin{corollary} Every homomorphism is a quasimorphism. In particular, the pullback of $\alpha \in \mathcal Q(H)$ by a homomorphism $f : G\to H$ satisfies $f^*\alpha \in \mathcal Q(G)$, whence
the functor $\mathcal Q$ is well-defined.
\end{corollary}
Since the pullback of a homogeneous function by a homomorphism is obviously again homogeneous we also deduce that the functor $\mathcal H: \mathfrak{Grp} \to \mathfrak{Vect}$ is well-defined. Next let us check that the categories $\mathfrak{QGrp}$ and $\mathfrak{HQGrp}$ are well-defined. For $\mathfrak{QGrp}$ this follows immediately from the fact that composition of maps is associative. Concerning $\mathfrak{HQGrp}$ we establish:
\begin{lemma} Let $f_1, f_2 \in {\rm Hom}_{\mathfrak{QGrp}}(G, H)$ and $g_1, g_2 \in  {\rm Hom}_{\mathfrak{QGrp}}(H, K)$. If $f_1 \sim f_2$ and $g_1 \sim g_2$ then $g_1 \circ f_1 \sim g_2 \circ f_2$.
\end{lemma}
\begin{proof} Let $\alpha \in H(K)$. By assumption the functions $b_H := g_1^*\alpha - g_2^*\alpha$ and $b_G:= f_1^*(g_2^*\alpha) -f_2^*(g_2^*\alpha)$ are bounded. Thus,
\begin{eqnarray*}
(g_1 \circ f_1)^*\alpha - (g_2 \circ f_2)^*\alpha &=& f_1^*(g_2^*\alpha + b_H) - f_2^*(g_2^*\alpha)\\
&=& b_H \circ f_1 + b_G,
\end{eqnarray*}
which is bounded.
\end{proof}
At this point we have established that all the functors and categories introduced in the introduction are well-defined. Note that for maps $f: G \to \R$ we currently have two different notions of quasimorphism: Let us temporarily call $f$ a quasimorphism of the first kind if its defect is bounded and a quasimorphism of the second kind if it pulls back all quasimorphisms of the first kind to such quasimorphisms. We also have according notions of equivalence of the first and second kind.
\begin{lemma} The two notions of quasimorphisms for maps $f: G \to \R$ coincide. Similarly, the two notions of equivalence coincide.
\end{lemma}
\begin{proof} Let $f$ be a quasimorphism of the first kind and $\alpha \in \mathcal H(\R)$. By Lemma \ref{sclStuff}.(iv) we have $\alpha = \lambda \cdot {\rm Id}$ for some $\lambda \in \R$, whence $f^*\alpha = \lambda \cdot f \in \mathcal H(G)$, showing that $f$ is also of the second kind. The converse is obvious, since $f^*{\rm Id} = f$. The proof concerning equivalence is similar.
\end{proof}
\begin{corollary} ${\rm Hom}_{\mathfrak{QGrp}}(G; \R) = \mathcal Q(G)$ and ${\rm Hom}_{\mathfrak{HQGrp}}(G; \R) = \mathcal H(G)$.
\end{corollary}
For our further study of homsets we introduce the abbreviations
\[
\widetilde{QQ}(G,H) := {\rm Hom}_{\mathfrak{QGrp}}(G; H), \quad QQ(G, H) := {\rm Hom}_{\mathfrak{HQGrp}}(G; H).
\]
We also keep the notation $\widehat{\alpha}$ to denote the homogenization of a quasimorphism $\alpha$. We can then reformulate Criterion \eqref{InjectivityHomMap} as follows:
\begin{lemma}\label{QoutAction} The map
\[
 \iota: QQ(G, H) \to {\rm Hom}_{\mathfrak Vect}(\mathcal H(G), \mathcal H(H)), \quad \iota(f)(\alpha) = \widehat{f^*\alpha},
\]
and hence also the map ${\rm QOut}(G) = QQ(G, G)^\times \to {\rm Aut}_{\mathfrak{Vect}}(\mathcal H(G))$ is injective.
\end{lemma}
We draw to immediate consequences. Firstly, we deduce the following claim from the introduction:
\begin{corollary}\label{OutQout} The natural map ${\rm Aut}_{\mathfrak Grp}(G) \to {\rm QOut}(G) $ factors through ${\rm Out}(G)$.
\end{corollary}
Secondly we observe that the inclusion $QQ(G, H) \hookrightarrow {\rm Hom}(\mathcal H(G), \mathcal H(H))$ equips the homset $QQ(G,H)$ with the structure of a vector space. The abelian group structure actually has an intrinsic representation: The sum $[f_1]\oplus[f_2]$ is represented by the pointwise product $g\mapsto f_1(g)f_2(g)$ (or, alternately, $g \mapsto f_2(g)f_1(g)$), the neutral element is represented by the constant map $g \mapsto e_H$ and 
the inverse of $f$ is represented by $g\mapsto f(g)^{-1}$. 
 \begin{lemma} The composition map $QQ(G, H) \times QQ(H,K) \to QQ(G,K)$ is bilinear.
\end{lemma}
\begin{proof} Let $g_1, g_2 \in \widetilde{QQ}(G, H)$, $h_1, h_2 \in \widetilde{QQ}(H,K)$ and $\alpha \in \mathcal H(K)$. Then there exist bounded functions $b_j(x)$, $j=1,2$ such that 
\begin{eqnarray*}
(h_1\circ (g_1g_2))^*\alpha(x) &=&(h_1^*\alpha)(g_1(x)g_2(x)) \\
&=&  (h_1^*\alpha)(g_1(x))+ (h_1^*\alpha)(g_2(x)) + b_1(x)\\
&=&  (h_1 \circ g_1)^*\alpha(x) +  (h_1 \circ g_2)^*\alpha(x) + b_1(x)\\
&=& [(h_1\circ g_1)(h_1 \circ g_2)]^*\alpha(x) + b_2(x), 
\end{eqnarray*}
hence 
\[[h_1]\circ ([g_1]\oplus[g_2]) =[h_1 \circ (g_1g_2)]= [(h_1\circ g_1)(h_1 \circ g_2)] =  ([h_1]\circ [g_1])\oplus([h_1] \circ [g_2]).\] Similarly, 
\begin{eqnarray*}
(h_1h_2\circ g_1)^*\alpha(x)
&=&  (h_1 \circ g_1)^*\alpha(x) +  (h_2 \circ g_1)^*\alpha(x) + b_1(x)\\
&=& [(h_1\circ g_1)(h_2 \circ g_1)]^*\alpha(x) + b_2(x).
\end{eqnarray*}
\end{proof}
Since finite biproducts clearly exist in $\mathfrak{QGrp}$ we deduce:
\begin{corollary}\label{CorAdditiveCategory} The category $\mathfrak{HQGrp}$ is an additive category.
\end{corollary}
\subsection{The group ${\rm QOut}(G)$}
In the sequel we will always consider ${\rm QOut}(G)$ as a subgroup of  ${\rm Aut}_{\mathfrak{Vect}}(\mathcal H(G))$ by means of Lemma \ref{QoutAction}. We will also denote by
\[{\rm qout}_G: {\rm Out}(G) \to {\rm QOut}(G)\]
the canonical map given by Corollary \ref{OutQout}. In general, this map is neither injective nor surjective. We will deal with the cases of amenable, respectively free groups in Section \ref{SecAmenable} and Section \ref{SecFreeQM} below. Before we turn to these computations we collect some properties of ${\rm QOut}(G)$ which can be derived purely formally.
\begin{proposition} ${\rm QOut}(G)$ acts continuously on $\mathcal H(G)$ with respect to the topology of pointwise convergence.
\end{proposition}
\begin{proof} The proof is just as for ${\rm Out}(G)$: If $\alpha_n \to \alpha$ pointwise in $\mathcal H(G)$, i.e., $\alpha_n(g) \to \alpha(g)$ for all $g \in G$, and $[\phi]\in {\rm QOut}(G)$, then $\phi^*
\alpha_n(g) =\alpha_n(\phi(g)) \to \alpha(\phi(g)) = \phi^*\alpha(g)$ for all $g \in G$.
\end{proof}
\begin{corollary} If $V < \mathcal H(G)$ is a subspace which is dense for the topology of pointwise convergence, then every $g \in {\rm QOut}(G)$ is uniquely determined by the restriction $g|_V: V \to  \mathcal H(G)$.
\end{corollary}
The corollary is particularly useful for computations in free groups, where we can use the dense subspace $\mathcal H^*(F)$ constructed in Theorem \ref{HStar}.

Recall that given a group $G$, a subgroup $H< G$ is called a \emph{retract} of $G$ if there exists a \emph{retraction} $r: G \to H$, i.e., a group homomorphism which is left-inverse to the inclusion $\iota_H: H \to G$. A left-inverse to the class $[\iota_H]$ in the category $\mathfrak{HQGrp}$ will be called a \emph{quasi-retraction}, and if such a quasi-retraction exists, then $H$ will be called a \emph{quasi-retract} of $G$. We observe:
\begin{lemma} If $H<G$ is a quasi-retract then the restriction map ${\rm res}_H^G: \mathcal H(G) \to \mathcal H(H)$, $f \mapsto f|_H$ is onto.
\end{lemma}
\begin{proof} If $r: G \to H$ is a quasimorphism whose class is a quasi-retraction, then for every $\alpha \in \mathcal H(H)$ we have $\alpha = (\widehat{r^*\alpha})|_H$, since both are homogeneous quasimorphisms of bounded distance.
\end{proof}

\begin{lemma} If $r: G \to H$ is a quasimorphism such that $[r]$ is a quasi-retraction, then 
\[r^*: QQ(H,H) \to QQ(G,G), [f] \mapsto [\iota_H \circ f \circ r]\]
is an injective algebra homomorphism.
\end{lemma}
\begin{proof} The map is well-defined and provides a homomorphism since $[r\iota_H] = [{\rm Id}_H]$. Now assume that $f \in {\ker(r^*)}$ and let $\psi \in \mathcal H(H)$. Let $\phi := r^*\psi  \in \mathcal H(G)$ so that $\psi = \phi|_{H}$. Then we have
\[\phi = (\iota_H \circ f \circ r)^*\phi = r^*f^*\iota_H^*\phi\Rightarrow \iota_H^*\phi = f^*\iota_H^*\phi \Rightarrow \psi = f^*\psi.\]
This shows that $[f]$ is trivial in $QQ(H,H)$, so $r^*$ is injective.
\end{proof}
\begin{corollary} For groups $H_1, H_2$ there is a canonical injection
\[{\rm QOut}(H_1) \times {\rm QOut}(H_2) \hookrightarrow {\rm QOut}(H_1 \times H_2).\]
\end{corollary}
\begin{proof}
The projections from $G = H_1 \times H_2$ to the factors is a retraction, hence induces injective maps $QQ(H_j) \to QQ(G)$. Since the images act on different factors, the product map 
\[QQ(H_1) \times QQ(H_2) \to QQ(G)\]
is still injective. Moreover, it maps $({\rm id}_{H_1}, {\rm id}_{H_2})$ to ${\rm id}_G$. It follows that it maps inverses to inverses and thus preserves the subgroups of invertible elements
\end{proof}
In general, the canonical map will not be surjective, e.g. the flip $(h_1, h_2) \mapsto (h_2, h_1)$ in ${\rm QOut}(H\times H)$ is not contained in its image. 

\subsection{The quasification functor}

Our next goal is to construct a more efficient model of the category $\mathfrak{HQGrp}$. To this end we construct a projection functor $Q: \mathfrak{HQGrp} \to \mathfrak{HQGrp}$, which is an equivalence of categories. This will reduce computations in $\mathfrak{HQGrp}$ to computations in $Q(\mathfrak{HQGrp})$. We need some basic results concerning kernels of quasimorphisms:
\begin{definition}
Let $\alpha: G \to \R$ be a  quasimorphism. A subgroup $N$ of $G$ is called a \emph{period subgroup} if $\alpha|_N$ is bounded. A quasimorphism is called \emph{aperiodic} if every period subgroup is trivial. If  $H$ is a quotient of $G$ with canonical projection $p: G \to H$, then $\alpha$ is said to \emph{factor through} $H$ if $\alpha = p^*\beta$ for some quasimorphism $\beta: H \to \R$. 
\end{definition}
We recall the following result from \cite{BSHLie}:
\begin{lemma} Let $\{e\} \to N \to G \to H \to \{e\}$ be a short exact sequence of groups and $\alpha: G \to \R$ be a homogeneous quasimorphism. Then the following are equivalent:
\begin{itemize}
\item[(i)] $\alpha|_N$ is bounded.
\item[(ii)] $\alpha|_N \equiv 0$.
\item[(iii)] $\alpha$ factors through $Q$.
\end{itemize}
\end{lemma}
Given $g \in G$ we denote by $N(g)$ the smallest normal subgroup of $G$ containing $g$. Then the last lemma implies (cf. \cite{BSHLie}):
\begin{proposition}
Let $G$ be a group and $\alpha: G \to \R$ a quasimorphism. Then there is a unique maximal period subgroup for $\alpha$, which is given by 
\[\ker(\alpha) = \{g \in G,|\, \alpha|_{N(g)} {\rm bounded}\}.\]
Moreover, the homogenization of $\alpha$ factors through an aperiodic quasimorphism on $G/{\ker(\alpha)}$.
\end{proposition}
We refer to $\ker(\alpha)$ as the \emph{kernel} of $\alpha$.
\begin{definition}
Let $G$ be a group. The normal subgroup
\[R_q(G) := \bigcap_{\alpha \in \mathcal H(G)} \ker(\alpha)\]
is called the \emph{quasi-radical} of $G$ and the quotient $Q(G) := G/R_q(G)$ is called the \emph{quasification} of $G$. The group $G$ is  \emph{quasi-separated} if $Q(G) = G$.
\end{definition}
Quasification has the following universal property:
\begin{proposition}\noindent \label{FunctorQ}
\begin{itemize} 
\item[(i)] Every homogeneous quasimorphism of $G$ factors through $Q(G)$ and $Q(G)$ is maximal with this property.
\item[(ii)] If $f: G \to H$ is any homomorphism, then there is a unique homomorphism $Q(f): Q(G) \to Q(H)$ such that the diagram
\[\begin{xy}\xymatrix{
G \ar[d]\ar[r]^f& H\ar[d]\\
Q(G)\ar[r]^{Q(f)}&Q(H)
}\end{xy}\]
commutes.
\item[(iii)] For all groups $G$ and homomorphisms $f$ we have $Q(Q(G)) = Q(G)$ and $Q(Q(f)) =f$.
\end{itemize}
\end{proposition}
\begin{proof} (i) holds by construction of $Q(G)$ and (ii) follows from Lemma \ref{sclStuff}.(i). For (iii) assume $x \in R_q(Q(G))$. Then  $f(x) = 0$ for all homogeneous quasimorphisms $f: Q(G) \to \R$. Let $p: G \to Q(G)$ denote the projection and choose $y \in G$ with $p(y) = x$; then $p^*f(y) = 0$ for all $f$ as above. By (i) this implies $F(y) = 0$ for every homogeneous quasimorphism on $G$, hence $y \in R_q(G)$. This implies that $x= p(y) = e$, hence $R_q(Q(G))$ is trivial.
\end{proof}
The proposition can be restated as saying that $Q$ is a projection functor from the category of groups to the full subcategory of quasi-separated groups. The main result of this subsection is the following theorem:
\begin{theorem}\label{QuasificationEquivalence}
Quasification extends to a functor $Q: \mathfrak{HQGrp} \to \mathfrak{HQGrp}$, which is a self-equivalence of categories. In particular, every isomorphism class in $\mathfrak{HQGrp}$ is represented by a quasi-separated group.
\end{theorem}
Let us start by extending Q to a functor on $\mathfrak{HQGrp}$.
\begin{proposition} Let $G, H$ be groups, $p_G : G \to Q(G)$, $p_H: H \to Q(H)$ the canonical projections and $f: G \to H$ a quasimorphism. Then there exists  a unique up to equivalence quasimorphism $Q(f): Q(G) \to Q(H)$ such that
\begin{eqnarray}\label{QuasiFunctoriality}
Q(f)(x) \in p_H (f(p_G^{-1}(x)).\end{eqnarray}
If $f$ is equivalent to a homomorphism, then so is $Q(f)$.
\end{proposition}
\begin{proof} Let $p_G(g_1) = p_G(g_2)$. Then $g_2 = g_1k$ for some $k \in G$ which is contained in the kernel of every homogeneous quasimorphism on $G$. Now assume that $\phi: H \to \R$ is a homogeneous quasimorphism; then
\[|f^*\phi(g_2) -f^*\phi(g_1)| = |f^*\phi(g_1k)- f^*\phi(g_1)| \leq |f^*\phi(k)| +D(f^*\phi),\]
and since $k \in \ker f^*\phi$ we obtain
\begin{eqnarray}\label{QuasiFunc1}
|f^*\phi(g_2) -f^*\phi(g_1)| &\leq& D(f^*\phi)
\end{eqnarray}
Now choose $Q(f): Q(G) \to Q(H)$ to be an arbitrary map subject to \eqref{QuasiFunctoriality} and let $g, h \in Q(G)$. By \eqref{QuasiFunctoriality} there exists elements $\hat{g}, \hat{h}, \widehat{gh} \in G$ in the respective $p_G$-fibers of $g, h, gh$ such that 
\[Q(f)(g) = p_H(f(\hat g)), \quad Q(f)(h) = p_H(f(\hat h)), \quad Q(f)(gh) = p_H(f(\widehat{gh})).\]
Let $\psi: Q(H) \to \R$ be a homogeneous quasimorphism and $\phi := p_H^*\psi$. Then
\begin{eqnarray*}
&&|Q(f)^*\psi(gh)-Q(f)^*\psi(g)-Q(f)^*\psi(h)|\\ &=& |\psi(p_H(f(\widehat{gh})))-\psi(p_H(f(\hat g)))-p_H(f(\hat h))|\\
&=& |f^*\phi(\widehat{gh})-f^*\phi(\hat{g})-f^*\phi(\hat{h})|\\
&\leq& |f^*\phi(\widehat{gh})-f^*\phi(\widehat{g}\widehat{h})|+ |f^*\phi(\widehat{g}\widehat{h})-f^*\phi(\hat{g})-f^*\phi(\hat h)|\\
&\overset{\eqref{QuasiFunc1}}\leq& 2D(f^*\phi),
\end{eqnarray*}
hence $Q(f)^*\psi$ is a quasimorphism. This shows that $Q(f)$ is a quasimorphism. It follows from \eqref{QuasiFunc1} that if $Q_1(f), Q_2(f)$ are two different homogeneous quasimorphisms satisfying \eqref{QuasiFunctoriality}, then for every homogeneous quasimorphism $\phi: Q(H) \to \R$ the difference $Q_1^*\phi- Q_2^*\phi$ is uniformly bounded by \eqref{QuasiFunc1}, whence $Q_1(f)$ and $Q_2(f)$ are equivalent. The last statement of the proposition follows from a fact that every homomorphism $G \to H$ descends to $Q(G) \to Q(H)$ by Proposition~\ref{FunctorQ} and the uniqueness part.
\end{proof}
We emphasize that given a quasimorphism $f$ the quasimorphism $Q(f)$ is defined only up to equivalence. Nevertheless we will abuse notation and write $Q(f)$ to denote any fixed choice of representative. With this abuse of notation understood, we have a well-defined map
\begin{equation}\iota_{G,H}: QQ(G, H) \to QQ(Q(G), Q(H)), \quad [f] \mapsto [Q(f)].\end{equation}
Now the key step in the proof of Theorem \ref{QuasificationEquivalence} is the following observation:
\begin{proposition}
The map $\iota_{G,H}$ is an isomorphism of abelian groups.
\end{proposition}
\begin{proof} We first show that $\iota_{G,H}$ is a morphism of abelian groups.  Let us fix a (set-theoretic) section $\sigma_G: Q(G) \to G$ of $p_G$. Then for any quasimorphism $f: G \to H$ a representative of the class $[Q(f)]$ is given by $Q(f)(x) = p_H(f(\sigma_G(x)))$. With this choice of representatives we see that for any pair 
$f_1, f_2: G \to H$ of quasimorphisms we have
\begin{eqnarray*}
 (Q(f_1) \cdot Q(f_2))(x) 
&=& p_H (f_1(\sigma_G(x)))\cdot p_H (f_2(\sigma_G(x)))\\
&=& p_H(f_1(\sigma_G(x))f_2(\sigma_G(x)))\\
&=&  p_H(f_1f_2(\sigma_G(x)))\\
&=& Q(f_1f_2)(x),
\end{eqnarray*}
showing that $[Q(f_1f_2)] = [Q(f_1)][Q(f_2)]$. Next let us compute the kernel of $\iota_{G,H}$. Suppose $[f] \in \ker(\iota_{G,H})$, i.e., $[Q(f)]$ is the class of the constant map $e$.
Let $\beta \in \mathcal H(H)$ and observe that $\beta = p_H^*\alpha$ for some $\alpha \in \mathcal H(Q(H))$. We then have $[Q(f)^*\alpha] = [e^*\alpha] = [0]$ and hence
\[
[f^*\beta] = [p_G^*Q(f)^*\alpha] = [0].
\]
Since $\beta$ was arbitrary,  we deduce that $[f]$ represents the trivial class, whence $\iota_{G,H}$ is injective. To see surjectivity fix a quasimorphism $g: Q(G) \to Q(H)$ and choose any function $f: G \to H$ satisfying 
\begin{equation}\label{LiftingQ}
p_H \circ f = g \circ p_G.
\end{equation}
We claim that $f$ is a quasimorphism. Indeed let $\beta \in \mathcal H(H)$ and choose  $\alpha \in \mathcal H(Q(H))$ with $\beta = p_H^*\alpha$. Then 
\[f^*\beta = (p_H \circ f)^*\alpha =p_G^*(g^*\alpha)\]
is a quasimorphism, establishing the claim. It then follows from \eqref{LiftingQ} that $[Q(f)] = [g]$.
\end{proof}
\begin{corollary}\label{Qs}
Let $G, H$ be groups. Then
\[QQ(G, H) \cong QQ(G, Q(H)) \cong QQ(Q(G), H) \cong QQ(Q(G), Q(H)).\]
\end{corollary}
\begin{proof} Since $Q^2 = Q$ we have $QQ(G, Q(H)) \cong QQ(Q(G), Q(H))$. A similar argument yields $QQ(Q(G), H) \cong QQ(Q(G), Q(H))$.
\end{proof}
This finishes the proof of Theorem \ref{QuasificationEquivalence}.

\subsection{Quasi-separated groups} In view of Theorem \ref{QuasificationEquivalence} the category of homogeneous quasi-groups is equivalent to its full subcategory of quasi-separated groups. We would thus like to understand the structure of quasi-separated groups. As far as amenable quasi-separated groups are concerned, it is easy to obtain a complete understanding:
 \begin{proposition}\label{AmenableQ} If $G$ is amenable, then $Q(G)$ is the quotient of the abelianization $G_{ab}$ of $G$ by its torsion subgroup. In particular, an amenable group is quasi-separated if and only if it is torsion-free abelian.
 \end{proposition}
 \begin{proof} If $G$ is amenable, then every homogeneous quasimorphism $f: G \to \R$ is a homomorphism by Lemma \ref{sclStuff}, hence factors through $G_{ab}$. Moreover, every homomorphism $G_{ab} \to \R$ factors through the torsion subgroup of $G_{ab}$. Thus $Q(G)$ is a quotient of $G_{ab}/{\rm Tor}(G_{ab})$. Then the proposition follows from the fact that homomorphisms into $\R$ separate points in a torsion-free abelian groups.
 \end{proof}
 On the other hand, there exist many non-abelian quasi-separated groups. For example, one can deduce from \cite{EpsteinFujiwara} that every torsion-free hyperbolic group is quasi-separated. More generally, one can consider the class of \emph{acylindrically hyperbolic groups}, which was introduced in \cite{Osin} where it is also shown to coincide with various other 
previously studied classes of groups with weak hyperbolicity properties. It comprises all non-elementary hyperbolic and relatively hyperbolic groups, all but finitely many mapping class groups and all outer automorphism groups of finitely-generated non-abelian free groups. 
  
We recall from \cite{Osin} that a group $G$ is called acylindrically hyperbolic if it admits an acylindrical isometric action on a Gromov-hyperbolic metric space $X$, which is non-elementary in the sense that the limit set of $G$ in $\partial X$ contains at least three points. Here, acylindricity of the action means that
$\forall \varepsilon > 0$ $\exists R, N >0$ $\forall x,y \in X$:
\[d(x,y) >R \Rightarrow \left|\left\{g \in G\mid \max\{d(x,gx), d(y, gy)\} < \varepsilon \right\}\right|<N
\]
Various equivalent characterizations of acylindrically hyperbolic groups are provided in \cite[Thm. 1.2]{Osin}. In particular, such groups contain proper infinite hyperbolically embedded subgroups. It then follows from \cite[Thm. 6.14]{DGO} that every acylindrically hyperbolic group $G$ contains a unique maximal finite normal subgroup $K(G)$. 
 \begin{proposition}\label{Osin}
 Let $G$ be an acylindrically hyperbolic group maximal finite normal subgroup $K(G)$. Then the quasimorphic radical of $G$ is given by
  $R_q(G) = K(G)$. In particular, $G$ is quasi-separated if and only if it does not contain any finite normal subgroup.
 \end{proposition}
 \begin{proof}[Proof (D. Osin)] Since $K(G)$ is amenable, every real-valued homogeneous quasimorphism on $G$ restricts to a homomorphism on $K(G)$, but since $K(G)$ is torsion every such homomorphism vanishes. Thus $K(G) \subseteq R_q(G)$ and it remains only to show that if $K(G) = \{e\}$ then for every normal subgroup $N \lhd G$ there exists a real-valued quasimorphism of $G$ which is unbounded on $N$. Thus let us fix a group $G$ acting acylindrically and non-elementarily on a hyperbolic space $X$ with $K(G) = \{e\}$ and let $N$ be a non-trivial normal subgroup. Then for every $g \in G$ we have $|gNg^{-1}\cap N| = |N| = \infty$, since there are no finite normal subgroups. According to \cite[Lemma 7.2]{Osin} this property (called \emph{s-normality} in \cite{Osin}) implies that $N$ acts non-elementarily on $X$. Combining this with \cite[Theorem 1.1]{Osin} we see that $N$
contains two (in fact, infinitely many) independent loxodromic elements $x,y$. The construction given in the proof of \cite[Theorem 1]{BestvinaFujiwara} then yields a non-trivial homogeneous quasimorphism on $G$ which does not vanish on the subgroup generated by $x$ and $y$, hence is unbounded on $N$.\end{proof}
 
Proposition \ref{Osin} provides a large supply of quasi-separable groups and shows thereby that our theory of quasimorphisms has some non-trivial content. Among the examples of quasi-separated groups covered by the proposition are not only torsion-free (relatively) hyperbolic groups, but also mapping class groups of closed surfaces of genus $\geq 3$ and outer automorphisms of free groups of finite rank $\geq 3$. 

\section{Quasioutomorphism groups of amenable groups}\label{SecAmenable}
The goal of this section is to establish the following result:
\begin{theorem}\label{AmenableMain1}
Let $G, H$ be amenable groups and assume that the abelianization $H_{ab}$ of $H$ has finite rank. Then 
\[QQ(G, H) \cong {\rm Hom}(G, H_{ab} \otimes_{\Z} \R).\]
In particular, if $G$ is  amenable and $r := {\rm rk}(G_{ab}) < \infty$, then 
\[{\rm QOut}(G) \cong {\rm GL}_r(\R).\]
\end{theorem}
We start from the following observation:
\begin{lemma}\label{TorsionFreeTarget}
Let $G$ be any group and $H$ be a torsion-free abelian group. Then there is an embedding
\[{\rm Hom}(G, H) \hookrightarrow QQ(G, H), \quad f \mapsto [f].\]
\end{lemma}
\begin{proof} Every $f \in {\rm Hom}(G, H)$ is a quasimorphism by Lemma \ref{sclStuff}. If $f_1, f_2$ are distinct homomorphisms then their classes in $QQ(G,H)$ are also distinct since $f_1 \neq f_2$ implies that there exists $h \in {\rm Hom}(H; \R)$ such that $f_1^*h \neq f_2^*h$. (Here we use, that homomorphisms into $\R$ separate points in a torsion-free abelian group.) Now distinct homomorphisms cannot be at bounded distance, so $f_1^*h$ and $f_2^*h$ (and consequently $f_1$ and $f_2$) are not equivalent.\end{proof}
\begin{lemma}\label{TensorTrick}
Let $H$ be an abelian group and denote by $\iota: H \to H \otimes_{\Z} \R$ the natural map. Then $[\iota]$ is invertible in $QQ(H, H \otimes_{\Z} \R)$. In particular, $H$ and $H \otimes_{\Z} \R$ are isomorphic as homogeneous quasigroups.
\end{lemma}
\begin{proof} Write $H$ additively and denote by $\lfloor\cdot\rfloor: \R \to \Z$ the floor function. Define a map
\[\{\cdot\}: H \otimes_{\Z} \R \to H, \quad \{h \otimes \lambda\} := \lfloor \lambda \rfloor \cdot h.\]
Then it is easy to check that $\{\cdot\}$ is a quasimorphism representing $[\iota]^{-1}$.
\end{proof}
\begin{corollary}\label{AmenablePre}
Let $G$ be any group and $H$ amenable. Then ${\rm Hom}(G, H_{ab} \otimes_{\Z} \R)$ embeds into $QQ(G, H)$.
\end{corollary}
\begin{proof} By Corollary \ref{Qs}, Lemma \ref{TensorTrick} and Proposition \ref{AmenableQ}
we have \begin{eqnarray*}QQ(G, H) &\cong& QQ(G, Q(H)) \cong  QQ(G, Q(H) \otimes_\Z \R) \cong QQ(G, H_{ab} \otimes_{\Z} \R).\end{eqnarray*}
Now apply Lemma \ref{TorsionFreeTarget}.
\end{proof}
Under additional assumptions on $G$ and $H$ we can in fact obtain an isomorphism, based on the following observation:
\begin{lemma}\label{HomAmenable}
Let $G$ be a group and $V$ be a finite-dimensional $\R$-vector space. Assume that every homogeneous quasimorphism on $G$ is a homomorphism. Then the map
\[{\rm Hom}(G, V) \to QQ(G, V), \quad f \mapsto [f]\]
is onto.
\end{lemma}
\begin{proof} Let $f_0: G \to V$ be a quasimorphism. We aim to construct a homomorphism $f: G \to V$ which is equivalent to $f_0$. If $V = \R$ then we can define $f$ to be the homogenization of $f_0$. If $V = \R^n$ we observe that the coordinates of $f_0$ are real-valued quasimorphisms $(f_0)_j: G \to \R$. If follows that the limit
\[
f(g) := \lim_{n \to \infty}\frac{f_0^n(g)}{n}
\]
exists and has coordinate functions given by $f_j = \widehat{(f_0)_j}$, the homogenization of $(f_0)_j$. Each $f_j$ is a homogeneous real-valued quasimorphism, hence a homomorphism by assumption. It follows that $f$ is a homomorphism, clearly equivalent to $f_0$.
\end{proof}
Combining this with Corollary \ref{AmenablePre} we get:
\begin{corollary}\label{AmenableMain}
Let $G$ be a group and $H$ an amenable group. Assume that every homogeneous quasimorphism on $G$ is a homomorphism and that $H_{ab}$ has finite rank. Then
\[QQ(G, H) \cong {\rm Hom}(G, H_{ab} \otimes_{\Z} \R).\]
\end{corollary}
Note that Theorem \ref{AmenableMain1} is a special case of Corollary \ref{AmenableMain}.

\section{Quasioutomorphism groups of free groups}\label{SecFreeQM}
\subsection{Embedding ${\rm Out}(F_n)$}
We have seen in the last section that quasimorphisms between amenable groups are essentially homomorphisms. Our next goal is to show that for non-amenable groups there may exist many proper quasimorphisms. We will focus on the case where $F$ is a finitely generated non-abelian free group. Throughout we denote by $n$ the rank of $F$ and fix a free generating set $S := \{a_1, \dots, a_n\}$. As before we identify elements of $F$ with reduced words over $S \cup S^{-1}$. If we want to emphasize the rank we also write $F_n$ for $F$. Our first observation concerns the injectivity of the canonical map
${\rm qout}_{F_n}: {\rm Out}(F_n) \to {\rm QOut}(F_n)$:
\begin{proposition}\label{prop-out-embedding}The map ${\rm qout}_{F_n}: {\rm Out}(F_n) \hookrightarrow {\rm QOut}(F_n)$ is injective.
\end{proposition}
\begin{proof} Assume that $f \in {\rm Aut}_{\mathfrak Grp}(F_n)$ represents a class $[f]$ in the kernel of the map ${\rm Out}(F_n) \to  {\rm QOut}(F_n)$. Write $f(a_j)$ as a reduced expression $f(a_j) \equiv x_jw_jx_j^{-1}$ with $w_j$ cyclically reduced. Consider the non-overlapping(!) $w_j$-counting quasimorphism $\phi^*_{w_j}$ and its homogenization $\widehat{\phi^*_{w_j}}$. Since $f^*\widehat{\phi^*_{w_j}} = \widehat{\phi^*_{w_j}}$ we deduce from Corollary \ref{phistarww} that
\begin{eqnarray*}
\widehat{\phi^*_{w_j}}(a_j) &=& \widehat{\phi^*_{w_j}}(f(a_j)) = \widehat{\phi^*_{w_j}}(x_jw_jx_j^{-1})\\
&=& \widehat{\phi^*_{w_j}}(w_j) = 1.
\end{eqnarray*}
This implies that for large $n$ we have $\phi^*_{w_j}(a_j^n) >0$, i.e., $w_j$ is a subword of $a_j^n$. Since $\widehat{\phi^*_{a_j^k}}(a_j) = \frac 1 k$ we deduce that $w_j=a_j$ is the only possibility.\footnote{Here it is important that we work with non-overlapping counting quasimorphisms, since $\widehat{\phi_{a_j^k}}(a_j) =1$.} It follows that for every $j=1, \dots, n$ there exists $x_j \in F_n$ and $n_j \in \mathbb N$ such that 
\begin{equation}
f(a_j) = x_ja_jx_j^{-1},
\end{equation}
and it remains to show only that $x_j$ is independent of $j$. Otherwise we may assume (after relabeling the generators) that $x_1 \neq x_2$. We write $x_1$ and $x_2$ as reduced expressions $x_1 = tu = t_1\cdots t_l u_1\cdots u_m$, $x_2=tv = t_1\cdots t_lv_1\cdots v_n$ with $t_j, u_j, v_j \in S\cup S^{-1}$, $m+n \geq 1$ and $u_1 \neq v_1$. The latter implies that the expressions $u^{-1}v = u_m^{-1}\cdots u_1^{-1}v_1\cdots v_n$ and $v^{-1}u$ are reduced. In particular, if we expand $f(a_1a_2)$ as
\[
f(a_1a_2) = f(a_1)f(a_2) = x_1a_1x_1^{-1} x_2a_2x_2^{-1} = tua_1u^{-1}va_2v^{-1}t^{-1},
\]
then the latter is a reduced expression. Let us abbreviate by $r := ua_1u^{-1}va_2v^{-1}$ the middle part of this expression. Then $r$ is reduced, and in fact cyclically reduced (since $v^{-1}u$ is reduced). We conclude in particular that $|r| \geq 4$. Moreover, Corollary \ref{phistarww} yields
\begin{eqnarray*}
\widehat{\phi^*_{r}}(a_1a_2) &=& \widehat{\phi^*_{r}}(f(a_1a_2)) = \widehat{\phi^*_{r}}(trt^{-1})\\
&=& \widehat{\phi^*_{r}}(r) = 1.
\end{eqnarray*}
By the same argument as above this implies that $r$ is a subword of $(a_1a_2)^N$ for some large $N$. Since we also require $\widehat{\phi^*_{r}}(a_1a_2) = 1$ the only possible choices are $r \in \{a_1, a_2, a_1a_2, a_2a_1\}$. This implies $|r| \leq 2$, which is a contradiction.
\end{proof}
Since ${\rm Out}(F_n)$ is non-amenable we conclude:
\begin{corollary}\label{CorNonAmenability}
The group ${\rm QOut}(F_n)$ is non-amenable.
\end{corollary}
\subsection{A criterion for quasiendomorphisms of free groups} Our further investigations of ${\rm QOut}(F_n)$ will be based on explicit constructions. The following criterion is a useful tool for proving that a map is a quasimorphism:
\begin{proposition}\label{thm:qmorphism:criterion:for:reduced:words}
Let $G$ be an arbitrary group and~$h \colon F_n \rightarrow G$ an arbitrary map. If there is a finite set~$E\subseteq G$ such that
\begin{enumerate}
\item \label{lem:prop:product}
for all words $w_1$ and~$w_2$, for which~$w_1w_2$ is a reduced word, we have $h(w_1 w_2) \in E h(w_1) E h(w_2) E$ and
\item \label{lem:prop:inverses} $\forall g\in F_n\colon \, h(g^{-1}) \in E h(g)^{-1} E$,
\end{enumerate}
then~$h$ is a quasimorphism.
\end{proposition}

\begin{proof}
Without loss of generality we assume that~$E$ is closed under taking inverses. Given $g, g' \in F_n$ we can always find reduced words $w, w', x$ so that $g \equiv wx$, $g' \equiv x^{-1}w'$ and the words $wx$, $x^{-1}w'$ and $ww'$ are all reduced. Now assume that $h$ satisfies the assumptions of the theorem. Then~$h(g g')= h(w w') \in E h(w) E h(w') E$. 
Furthermore~$h(g) \in   E h(w)E h(x)E$ which implies that~$h(w) \in  E h(g) Eh(x)^{-1}E $. Similarly we conclude~$h(w') \in E h(x)E  h(g')  E $. Assembling the three statements we obtain~$h(g g') \in E^2  h(g)E h(x)^{-1} E^3  h(x) E h(g') E^2$. If we define $\widetilde{E} := (E\cup\{e\})^3$ then this implies $h(gg') \in \widetilde{E}h(g) \widetilde{E} h(x)^{-1}  \widetilde{E}  h(x)  \widetilde{E} h(g')  \widetilde{E}$, so the proposition follows from Lemma \ref{lem:qmorphism:criterion}.
\end{proof}
As a first illustration of our criterion, we describe a class of quasiendomorphisms of free groups which change a word by local manipulation. Given $k >0$, let us refer to a map $f \colon (S\cup S^{-1})^k_\text{red} \rightarrow S^{\ast}$  from the reduced words over~$S\cup S^{-1}$ of length~$k$ to words of arbitrary length as a local transformation of length $k$ provided $f(u^{-1}) = f(u)^{-1}$ for every reduced word~$u$ of length~$k$. Given such a map we define~$\phi_f \colon F \rightarrow F$ to be the map that maps a reduced word~$w= w_1\dots w_n$ over~$F$ to the word~$\phi_f(w) = f(w_1\dots w_k) f(w_2\dots w_{k+1}) \dots f(w_{n+1-k}\dots w_n)$ if~$n\geq k$ and to the empty word otherwise.
\begin{lemma} For every local transformation $f$ the map~$\phi_f$ is a quasimorphism.
\end{lemma}
\begin{proof}
We are going to show that the map~$\phi_f$ satisfies the assumptions of Proposition~\ref{thm:qmorphism:criterion:for:reduced:words}. Concerning the first condition, let~$E = (\{e\} \cup f((S\cup S^{-1})^k_\text{red}))^{k}$. Let~$w = w_1\dots w_t$ and~$w' = w_{t+1} \dots w_n$ be words in the free group~$F$ such that~$ww'$ is reduced. Then \[\phi_f(ww') = f(w_1\dots w_k) f(w_2\dots w_{k+1}) \dots f(w_{n+1-k}\dots w_n),\] which, by definition of~$E$, is a word that is contained in~$f(w)Ef(w') =$ 
\[f(w_1\dots w_k)  \dots f(w_{t+1-k}\dots w_t) E f(w_{t+1}\dots w_{t+1+k}) \dots f(w_{t+1-k}\dots w_t).\] Concerning the second property, observe that for a reduced word~$w$ we have~\[f(w) = f(w_1\dots w_k) f(w_2\dots w_{k+1}) \dots f(w_{n+1-k}\dots w_n),\] while~\begin{eqnarray*}
f(w^{-1}) &=& f(w_n\dots w_{n+1-k}) f(w_{n-1}\dots w_{n-k}) \dots f(w_{k}\dots w_1)\\ &=& f(w_{n+1-k}\dots w_n)^{-1} \dots f(w_1\dots w_k)^{-1}.\end{eqnarray*}
This finishes the proof of the lemma.
\end{proof}
\begin{definition} Given a local transformation $f$ the quasimorphism $\phi_f$ is called the \emph{local quasimorphism} modeled according to $f$.
\end{definition}
Informally, a local quasimorphism maps a reduced word by manipulating every letter in a way that only takes into account the~$k-1$ letters that follow. One could try to construct a larger class of quasimorphisms by also allowing for a finite look-back, effectively taking into account the previous~$k'$ letters for some fixed~$k'$. However, up to equivalence this will not produce any new examples. Indeed, considering a look-back and a look-ahead of~$k$ we obtain a map given by~$\phi_f(w) = f(w_1\dots w_{k+1} \dots w_{2k+1}) \dots f(w_{i-k}\dots w_{i}\dots w_{i+k}) \dots f(w_{n-2k}\dots w_{n-k}\dots w_n)$, and this yields a function also that can also be interpreted as a function with a lookahead of~$2k$.

\subsection{Torsion in ${\rm QOut(F_n)}$}
Our next goal is to study the torsion of ${\rm QOut(F_n)}$. Recall that the wobbling group $W(\Z)$ is the group of permutations of~$\mathbb{Z}$ for which the distance of every integer to its image is bounded. Clearly every finite group is a subgroup of the wobbling group. We are going to establish the following result:
\begin{theorem}\label{thm:wobbling:embedding:and:consequences} For every $n \geq 2$ the wobbling group $W(\Z)$ embeds into ${\rm QOut}(F_n)$. In particular, for all such $n$ the group ${\rm QOut}(F_n)$ is uncountable and contains torsion elements of any given order.
\end{theorem}
Note that, in stark contrast to the theorem, the maximum order of torsion elements in~${\rm Out}(F)$ is bounded~\cite{MR1655470}. In order to prove the theorem we are going to describe another class of quasi-endomorphisms of $F_n$, which do not just modify words locally. For the purposes of the proof we will use the following embedding of $W(\mathbb Z)$. We can embed $\mathbb{Z}$ into~$\mathbb{N}$ by mapping~$i$ to~$2i+2$ for~$i\geq 0$ and to~$(-2) i-1$ for~$i<0$. 
This induces an embedding of $W(\Z)$ into the monoid 
\[
B(\mathbb N_0) := \{\sigma: \mathbb N_0 \to \mathbb N_0\,|\, \sigma(0) = 0,  \max\{|\sigma(k)-k|\mid k\in \mathbb{N}\}<\infty\}.
\]
We deduce that it suffices to construct an injective monoid homomorphism from $B(\mathbb N_0)$ to~$QQ(F_n, F_n)$. Thus let $\sigma \in B(\mathbb N_0)$ and define a map $\pi_\sigma: F_n \to F_n$ as follows: Given any $g \in F_n$ we can represent $g$ uniquely as $w_1a_n^{i_1}w_2 \cdots w_{l-1}a_n^{i_{l-1}}w_l$ with $w_j \in F_{n-1}$, $w_2, \dots, w_{l-1} \neq \varepsilon$ and $i_j \in \Z \setminus\{0\}$. We then define
\[
\pi_\sigma(g) = w_1a_n^{i_1'}w_2 \cdots w_{l-1}a_n^{i_{l-1}'}w_l,
\]
where
\[i'_k = 
\begin{cases} 
\sigma(i_k)  &\text{if~$i_k >0$,}  \\
-\sigma(-i_k)  &\text{if~$i_k < 0$} .\end{cases}\] 
Informally, we replace every occurrence of $a_n^i$ or $a_n^{-i}$ between two powers of other letters by $a_n^{\sigma(i)}$, respectively~$a_n^{-\sigma(i)}$. Applying Proposition~\ref{thm:qmorphism:criterion:for:reduced:words} with
\[
E := \big\{a_n^i\mid |i| \leq \max\{|\sigma(k)-k|\mid k\in \mathbb{N}\}\big\}
\] 
shows that~$\pi_\sigma$ is a quasimorphism for every $\pi \in B(\mathbb N_0)$, and evidently the map $\sigma\mapsto \pi_\sigma$ is a monoid homomorphism. Note that this quasimorphism is equivalent to a local quasimorphism only if $\sigma$ has finite support. In order to establish the embedding $W(\Z) \hookrightarrow {\rm QOut}(F_n)$ it remains to show only that the map $B(\mathbb N_0) \to {\rm QOut}(F_n)$ given by $\sigma \mapsto \pi_\sigma$ is injective.

\begin{proof}[Proof of Theorem~\ref{thm:wobbling:embedding:and:consequences}] 
Let $\sigma, \sigma' \in B(\mathbb N_0)$ be distinct. We claim that the functions~$\pi_\sigma$ and $\pi_{\sigma'}$ are not equivalent. Suppose~$\sigma(i) = j>  \sigma'(i)$. Let~$\varphi_{a_n^j}$ be the counting morphism that counts the number of occurrences of~$a_n^j$ and denote $w := a_1a_n^ia_1$. Then
\[
\pi_\sigma^*\varphi_{a_n^j}(w^k) = \varphi_{a_n^j}(a_1a_n^ja_1)^k = k,
\]
whereas $\pi_{\sigma'}^*\varphi_{a_n^j}(w^k) = 0$. Thus $\pi_\sigma^*\varphi_{a_n^j}(w^k) \not \sim \pi_{\sigma'}^*\varphi_{a_n^j}$ and hence $\pi_\sigma \not \sim \pi_{\sigma'}$, finishing the proof.
\end{proof}
\begin{remark} Denote by ${\rm QOut}_{(k)}(F_n)$ the characteristic subgroup of ${\rm QOut}(F_n)$ generated by all elements of order at most $k$ and by ${\rm QOut}_{({\rm fin})}(F_n)$ the characteristic subgroup generated by all elements of finite order. By Theorem \ref{thm:wobbling:embedding:and:consequences} these groups are non-trivial and we have a chain
\[{\rm QOut}_{(2)}(F_n) \lhd {\rm QOut}_{(3)}(F_n) \lhd \dots\lhd {\rm QOut}_{({\rm fin})}(F_n)\lhd {\rm QOut}(F_n).\]
of characteristic subgroups. It would be interesting to know, whether any of these inclusions are proper.
\end{remark}
\subsection{Free groups are not quasi-Hopfian}
We now provide a first application of the two classes of quasioutomorphisms of free groups constructed so far. Recall that every finitely-generated free group $G = F_n$ and every finitely generated free abelian group $G = \Z^n$ is \emph{Hopfian} in the sense that every epimorphism $G \to G$ is an isomorphism. 
\begin{definition} A quasi-separated $G$ is called \emph{quasi-Hopfian} if there does not exist a normal subgroup $N \lhd G$ and a surjective quasimorphism $f: G \to G$ with $f(N) = \{e\}$.
\end{definition}
By definition, every quasi-separated quasi-Hopfian group is Hopfian. Recall that by Proposition~\ref{AmenableQ} an amenable group $G$ is quasi-separated if and only if it is of torsion-free abelian. If such a group is finitely-generated then it is not only Hopfian, but even quasi-Hopfian:
\begin{proposition}\label{prop:torsion:free:abelian:qhopf} If $G$ is a torsion-free abelian group of finite rank, then $G$ is quasi-Hopfian.
\end{proposition}
\begin{proof} Assume that for some $n\geq 1$ we could find  a surjective quasimorphism $f: \Z^n \to \Z^n$ which vanishes on a normal subgroup $N \lhd \Z^n$. Then we obtain a surjective quasimorphism $g: \Z^k \oplus T \to \Z^n$ for some $k < n$ and a finite group $T$. If we denote by $\iota: \Z^n \to \R^n$ the canonical inclusion, then $\iota \circ g: \Z^k \oplus T \to \R^n$ is a quasimorphism, and by the proof of Lemma \ref{HomAmenable} it is at bounded distance from a homomorphism $g': \Z^k \oplus T \to \R^n$. It follows that there exists a $k$-dimensional subspace $V \subseteq \R^n$  such that~$g'(\Z^k \oplus T) \subseteq V$ and hence every $x\in \iota(f(Z^n))= \iota(g(Z^k \oplus T))$ is at bounded distance from $V$. But this contradicts the surjectivity of~$f$.
\end{proof}

On the contrary we can show that free groups of rank at least~$4$ are not quasi-Hopfian, which answers a question asked to us by Misha Kapovich. To show this, we first construct a surjective quasimorphism from~$F_{n-1}$ to~$F_{n}$ for~$n\geq 4$.

\begin{lemma}
For~$n\geq 4$ there exists a surjective quasimorphism from~$F_{n-1}$ to~$F_n$.

\end{lemma}

\begin{proof}
Given~$n\geq 4$ we denote by~$S = \{a_1, \dots, a_n\}$ a basis of~$F_n$. Let~$S' = \{a_2,\dots,a_{n}\}$. We start from a map $f \colon ((S'\cup S'^{-1})^2_\text{red} \rightarrow S^{\ast}$ as follows: We define $f(a_2a_3) = a_1a_3^{-1}$ and $f(a_3^{-1}a_2^{-1}) = a_1^{-1}a_2$. For all other reduced words $s_1s_2$ of length $2$ we define $f(s_1s_2) = s_1$. The map $f$ is not quite a local transformation as defined above since the condition $f(w^{-1}) = f(w)^{-1}$ is violated. Nevertheless, we can consider the map $\phi_f: F_{n-1} \to F_n$ given by
\[\phi_f(s_1 \cdots s_l) = f(s_1s_2)f(s_2s_3) \cdots f(s_{l-1}s_l).\]

It turns out to be convenient to modify this slightly and to define $\phi: F_{n-1} \to F_n$ by $\phi(s_1\cdots s_l) = \phi_f(s_1\cdots s_l)s_l$. Then $\phi$ has the effect of  replacing all occurrences of $a_2a_3$ (respectively $(a_3a_2)^{-1}$) by $a_1$ (respectively $a_1^{-1}$). It is immediate from this description that $\phi$ satisfies the conditions of Proposition \ref{thm:qmorphism:criterion:for:reduced:words}, hence defines a quasimorphism. 

Let~$\sigma$ be the map in~$B(\mathbb N_0)$ given by~$\sigma(0) = 0$ and~$\sigma(k)= k-1$ for~$k>0$. Let~$\pi_\sigma$ be the quasimorphism induced by this map, as defined in the previous section (i.e.,~$\pi_\sigma$ replaces all powers of~$a_n$ by powers of~$a_n$ whose absolute value is smaller by exactly one).
Then~$\pi \circ \phi$ is a quasimorphism which we claim to be surjective. Indeed, we can write every reduced word~$w$ over~$S\cup S^{-1}$ in the form~$a_{n}^{t_0}s_1a_{n}^{t_1}\cdots s_la_{n}^{t_l}$
with~$t_j \in \Z$ for all~$j\in \{0,\ldots l\}$ and~$s_j \in (S\cup S^{-1})\setminus\{a_n,a_n^{-1}\}$ for all~$1 \leq j\leq l$. We consider a word~$w'$ of the form~$a_{n}^{t'_0}s_1a_{n}^{t'_1}\cdots s_la_{n}^{t'_l}$
with~$t_j' = t_j+1$ if~$t_j \geq 0$ and~$t_j = t_j-1$ if~$t_j <0$ for all~$j\in \{0,\ldots l\}$. Note that~$w'$ is a preimage of~$w$ under~$\pi$. Now we replace in~$w'$ all occurrences of~$a_1$ by~$a_2a_3$ and all occurrences of~$a_1^{-1}$ by~$(a_3a_2)^{-1}$ obtaining a word~$w''$, which is reduced since~$n\geq 4$. Then~$w''$ is a preimage of~$w'$ under~$\phi$ and thus~$w''$ is a preimage of~$w$ under~$\pi \circ \phi$.
\end{proof}

We remark that no  surjective quasimorphism exists from~$F_1$ to~$F_2$ since any such quasimorphism would yield a surjective quasimorphism from~$F_1= \mathbb{Z}$ to~$\mathbb{Z}_2$, the abelianization of~$F_2$. Such a map cannot exists by the same arguments as those used in the proof of Proposition~\ref{prop:torsion:free:abelian:qhopf}.

\begin{theorem} For $n \geq 4$ the group $F_n$ is not quasi-Hopfian.
\end{theorem}
\begin{proof}

Given~$n\geq 4$ we denote by~$S = \{a_1, ..., a_n\}$ a basis of~$F_n$. The homomorphism~$\psi$ that maps~$a_1$ to~$\{e\}$ and fixes all other generators maps the normal subgroup generated by~$a_n$ to~$\{e\}$. Let~$\tau$ be a quasimorphism that maps~$F_{n-1}$ with generators~$\{a_2, ...,a_{n}\}$ surjectively to~$F_n$ mapping~$e$ to~$e$. Such a quasimorphism exists by the previous lemma. Consider the quasimorphism~$\tau \circ \psi$. This quasimorphism is surjective and maps the normal subgroup generated by~$a_1$ to~$\{e\}$, showing that~$F_n$ is not quasi-Hopfian.
\end{proof}

\subsection{Transitivity properties of ${\rm QOut}(F_n)$}\label{SecWordExchange} In order to study transitivity properties of ${\rm QOut}(F_n)$ we are going to introduce a third class of quasioutomorphism of $F_n$, whose effect on reduced words is given by replacing certain subwords.
\begin{definition}
Given a reduced word~$w$ over $S \cup S^{-1}$, a \emph{decomposition of~$w$} is a sequence of reduced words~$(u_1,\dots, u_t)$ such that~the concatenation $u_1 \dots u_t$ is reduced and coincides with $w$. Given a set ~$W = \{w_1,\dots,w_k\}$ of independent reduced words the \emph{number of occurrences of words from~$W$} in the decomposition~$(u_1,\dots,u_t)$ is the number of integers~$i\in \{1,\dots,t\}$ for which~$u_i\in W \cup W^{-1}$. A decomposition $(u_1,\dots,u_t)$ is called \emph{$W$-maximal} if this number is maximal.
A decomposition~$(u_1,\dots, u_{t+1})$ is a simple refinement of~$(u'_1,\dots,u'_t)$ if there is a~$k$ such that~$u_i = u'_i$ for~$i< k$,~$u_ku_{k+1} = u_{k'}$ and~$u_{i+1} = u'_{i}$ for~$i>k$. A decomposition is a refinement of another decomposition if it is obtained by repeated simple refinement steps. We remind the reader that the notion of independent words was defined in Definition \ref{DefIndependent}.
\end{definition}
\begin{lemma}
Given a set of independent words~$W = \{w_1,\dots,w_k\}$ and a word~$w$, there is a unique $W$-maximal decomposition $(u_1, \dots, u_t)$ of~$w$ that is minimal with respect to refinement among all $W$-maximal decompositions.
\end{lemma}
\begin{proof} The existence of a minimal $W$-maximal decomposition is obvious. Let~$u = (u_1,\dots,u_t)$ and~$u' = (u'_1,\dots,u'_{t'})$ be two $W$-maximal decompositions and assume that both are minimal with respect to refinement among such decompositions. By induction it suffices to show that~$u_1 = u'_1$. Suppose otherwise. Since the words in~$W$ are independent, this implies that~$u_1$ is not in~$W \cup W^{-1}$ or~$u_2$ is not in~$W \cup W^{-1}$. Without loss of generality we assume the former. Since~$u$ is minimal with respect to refinement,~$u_2$ is in~$W \cup W^{-1}$. If~$u'_1 \in W \cup W^{-1}$ then~$u'_1$ is shorter than~$u_1$ and thus~$u_1$ can be refined into two words, one of which is in~$W \cup W^{-1}$, increasing the number of occurrences of words from~$W$ in~$u$ and yielding a contradiction. If~$u'_1$ is not in~$W \cup W^{-1}$ we can without loss of generality assume that~$u'_1$ is shorter than~$u_1$. This implies~$u'_2$ is in~$W \cup W^{-1}$ and, due to independence, the length of~$u'_1u'_2$ is at most the length of~$u_1$. Again, we can split~$u_1$ into at most 3 words increasing the number of occurrences of words from~$W$ in~$u$ and yielding a contradiction.
\end{proof}
In the sequel we refer to this unique decomposition $(u_1, \dots, u_t)$ as the \emph{$W$-decomposition} of $w$.
\begin{definition} Let $W := \{w_1, w_2\}$ be a set consisting of two independent words which share the same initial letter and also share the same final letter. Given a reduced word $w \in F_n$ with $W$-decomposition $(u_1, \dots, u_t)$, define a new word $w'$ as $w' := u_1'\cdots u_t'$, where
 \[u'_i := 
\begin{cases} 
w_2  &\mbox{if } u_i = w_1 \\
w_1  &\mbox{if } u_i = w_2 \\
{w_2}^{-1}  &\mbox{if } u_i = {w_1}^{-1} \\
{w_1}^{-1}  &\mbox{if } u_i = {w_2}^{-1} \\
u_i   &\mbox{otherwise}.\end{cases}\]
Then the map $f_{w_1, w_2}: F_n \to F_n$ given by $f_{w_1, w_2}(w) = w'$ is called the \emph{word replacement quasimorphism} associated with the pair $\{w_1, w_2\}$.
\end{definition}
Note that it follows from Proposition \ref{thm:qmorphism:criterion:for:reduced:words} that $f_{w_1, w_2}$ is indeed a quasimorphism, since we can choose $E$ to be the set of all words of length at most the maximal lengths of $w_1$ and $w_2$. Informally, $f_{w_1, w_2}$ switches occurrences of $w_1^{\pm 1}$ with occurrences of $w_2^{\pm 1}$. Since the words that are replaced start and end with the same letters, word replacement quasimorphisms map reduced words to reduced words. Moreover $f_{w_1, w_2}$ maps $\{w_1, w_2\}$-decompositions to $\{w_1, w_2\}$-decompositions. It follows that $f_{w_1, w_2}^2 = {\rm Id}$, whence
\begin{equation}\label{WRQ}
f_{w_1, w_2} \in {\rm QOut}_{(2)}(F_n).
\end{equation}
Note that in order for $\{w_1, w_2\}$ to be independent it is necessary that $w_1$ and $w_2$ are self-independent. If their length is at least $2$, then this implies that their respective initial letters are different from their final letters.

We now study the action of word replacement quasimorphisms on counting quasimorphisms.
By construction of $f_{w_1, w_2}$ the number of occurrences of the word~$w_2$ in a word~$w$ is no smaller than the number of occurrences of~$w_1$ in $f_{w_1, w_2}(w)$. Since $f_{w_1, w_2}$ is an involution, the same argument shows that the number of occurrences of~$w_1$ in the $f_{w_1, w_2}(w)$ is no smaller than the number of occurrences of the word~$w_2$ in a word~$w$. This shows that
\begin{equation}\label{SwapWords}
f^*_{w_1, w_2}(\phi_{w_1}) = \phi_{w_2}, \quad f^*_{w_1, w_2}(\phi_{w_2}) = \phi_{w_1}.\end{equation}
We record the following consequence of  \eqref{WRQ} and \eqref{SwapWords} for later use:
\begin{lemma}\label{ExchangeCounting} Let $w_1$ and $w_2$ be self-independent words which share the same initial letter and also share the same final letter. If the set $\{w_1, w_2\}$ is independent, then $\widehat{\phi_{w_1}}$ and $\widehat{\phi_{w_2}}$ are contained in the same ${\rm QOut}_{(2)}(F_n)$-orbit in $\mathcal H(F_n)$.
\end{lemma}
The lemma motivates the study of pairs of independent words in $F_n$ with respect to our fixed generating set $S$.
\begin{lemma}\label{lem:arbitrary:to:aibl}
Let $w$ be a self-independent word of length $\geq 2$ with initial letter $a \in S\cup S^{-1}$ and final letter~$b \in S\cup S^{-1}$. There exist positive integers~$i,j$ such that~$\{w,a^i b^j\}$ is a set of independent words.
\end{lemma}
\begin{proof}
If~$w$ is in~$a^{\ast} b^{\ast}$, there is nothing to show. Now suppose otherwise. Let~$\ell$ be the length of~$w$. We claim that~$\{w,w'\}$, with~$w'= a^{\ell} b^{\ell}$, is a set of independent words. If~$w$ contains a letter that is not in~$\{a,b,a^{-1},b^{-1}\}$ this is obvious, so we suppose otherwise.
Since~$w$ is not in~$a^{\ast} b^{\ast}$ there are reduced words~$u,v,x$ such that~$w = aub^{\varepsilon} va^{\varepsilon'}x b$ with~$\varepsilon,\varepsilon' \in \{-1,1\}$. Since~$w'$ does not contain inverses of~$a$ or~$b$ and is longer than~$w$, the words~${w'}^{-1}$ and~$w$ do not overlap. We argue that no prefix of~$w$ is a postfix of~$w'$. The fact that no postfix of~$w$ is a prefix of~$w'$ follows by symmetry.
If a postfix of~$w'$ were a prefix of~$w$ it would have length at most~$\ell$. Thus it would be of the form~$b^{i'}$ for some~$i'\leq \ell$. However, since~$w$ starts with~$a$ it cannot be a prefix of~$w$.
\end{proof}
\begin{lemma}\label{lem:abibia:to:aibl}
For every pair of distinct elements $a, b \in S$ and all positive integers~$i,j>0$  the set~$\{ab^{-1}a^{-1}b,a^i b^j\}$ is a set of independent words.
\end{lemma}
\begin{proof}
By similar arguments as in the previous lemma, it suffices to argue that no prefix of~$ab^{-1}a^{-1}b$ is a postfix of~$a^i b^j$. Suppose~$p$ is a prefix of~$ab^{-1}a^{-1}b$. If~$p$ has length 1 then~$p = a$ and~$p$ is not a postfix of~$a^i b^j$. If~$p$ has a length that is greater than 1 then~$p$ contains~$b^{-1}$ and is thus not a postfix of~$a^i b^j$ either.
\end{proof}
\begin{corollary}\label{Cor2orbits} Let $a,b \in S$ be two letters with $a \not \in \{b, b^{-1}\}$ and let
\[\Phi(F_n, S) := \{\widehat{\phi_{w}}\,|\, w \text{ self-independent, reduced word over}\,S\} \subseteq \mathcal H^*(F) \subseteq \mathcal H(F).\] 
Then every ${\rm QOut}_{(2)}(F_n)$-orbit that intersects  
$\Phi(F_n, S)$ contains either $\widehat{\phi_{a}}$ or $\widehat{\phi_{ab}}$. In particular, the number of such orbits is at most $2$.
\end{corollary}
\begin{proof} Let $w$ be a self-independent reduced word over $S$ of length $\geq 2$ with initial letter $a$ and final letter $b\neq a$. Then by combining Lemma \ref{ExchangeCounting}, Lemma \ref{lem:arbitrary:to:aibl} and Lemma \ref{lem:abibia:to:aibl} the quasimorphism $\widehat{\phi_w}$ is in the same ${\rm QOut}_{(2)}(F_n)$-orbit as $\widehat{\phi_{ab}}$.  We deduce that every orbit admits either a representative of the form $\widehat{\phi_{ab}}$ for letters $a,b \in S$ with $a \not \in \{b, b^{-1}\}$ or a representative of the form $\widehat{\phi_a}$ for some $a \in S$. 

In order to show that all these quasimorphisms are in the same ${\rm QOut}_{(2)}(F_n)$-orbit as either $\widehat{\phi_{ab}}$ or $\widehat{\phi_a}$ for some fixed pair $\{a,b\} \subseteq S$ we consider the subgroup $\Gamma < {\rm Out}(F_n)$ generated by those Nielsen transformations which we denoted $P_1, P_2, I$ in Section \ref{SecNielsen}. Since $\Gamma$ is generated by involutions we have $\Gamma < {\rm QOut}_{(2)}(F_n)$. Now the corollary follows from the fact that $\Gamma$ acts transitively on $S$ and also on the set 
\[(S \cup S^{-1})^2 \setminus \{(s, r)\,|\,s  \in \{r, r^{-1}\}\},\]
and that $A^*\widehat{\phi_w} = \widehat{\phi_{A^{-1}w}}$ for $A \in \{P_1, P_2, I\}$ and hence for all $A \in \Gamma$.
\end{proof}
We do not know whether there is an element of ${\rm QOut}_{(2)}(F_n)$ or even ${\rm QOut}(F_n)$ which maps the orbits of $\phi_a$ to the orbit of $\phi_{ab}$ for $a, b$ as in the corollary. However, we can show that the $\phi_{ab}$-orbit is contained in the orbit \emph{closure} of the span of the $\phi_{a}$ orbit. This is based on the following observation.
\begin{lemma}\label{lem:behviour:of:a:under:replacement}
If~$\{w_1, w_2\}$  a set of independent words that start with the same letters and end with the same letters then 
\[f_{w_1, w_2}^*\phi_{a} = \phi_{a} + (\phi_a(w_2) -\phi_a(w_1)) (\phi_{w_1}- \phi_{w_2}).\]
\end{lemma}
\begin{proof} Comparing the number of occurrences of $a$ in $f_{w_1, w_2}(w)$ with occurrences in $w$ we see that for every occurrence of $w_1$ in $w$, $\#_a(w_1)$ copies of $a$ get removed, whereas $\#_a(w_2)$ copies of $a$ get added. Similarly, for every occurrence of $w_2$ in $w$, $\#_a(w_2)$ copies of $a$ get removed, whereas $\#_a(w_1)$ copies of $a$ get added. Combining this with a similar count for $a^{-1}$ the lemma follows.
\end{proof}
\begin{corollary}\label{OrbitClosure} Let $a,b \in S$ with $a \not \in \{b, b^{-1}\}$. Then
\[
\widehat{\phi_{ab}}\in \overline{{\rm span}({\rm QOut}_{(2)}(F_n).\widehat{\phi_a}}).\]
\end{corollary}
\begin{proof} We consider the words $w = ab^{-1}a^{-1}b$ and $w_k = ab^{-k}aba^{-1}b$ with~$k\geq 1$. Then for every $k\geq 1$ the set $W_k := \{w, w_k\}$ is a set of independent words with the same initial/final letters, which hence gives rise to an exchange quasimorphism $f_{w_k, w}$. Applying Lemma \ref{lem:behviour:of:a:under:replacement} and passing to homogenizations yields
\[
f_{w_k, w}^*\widehat{\phi_{a}} = \widehat{\phi_{a}} +  \widehat{\phi_{ab^{-1}a^{-1}b}} -  \widehat{\phi_{ab^{-k}aba^{-1}b}}.
\]
Note that by Lemma \ref{lem:abibia:to:aibl} the words $ab^{-1}a^{-1}b^{-1}$ and $ab$ are independent, hence by Lemma \ref{ExchangeCounting} there exists $g \in {\rm QOut}_{(2)}(F_n)$ such that $g^*\widehat{\phi_{ab^{-1}a^{-1}b}} = \widehat{\phi_{ab}}$. Let us define $g_k := f_{w_k, w}g \in {\rm QOut}_{(2)}(F_n)$. Then
\[
\widehat{{\phi}_{ab}} = g^*\widehat{\phi_{ab^{-1}a^{-1}b}} = g^*(f_{w_k, w}^*\widehat{\phi_{a}} - \widehat{\phi_{a}}+ \widehat{\phi_{ab^{-k}aba^{-1}b}}) = g_k^*\widehat{\phi_a} - g^*\widehat{\phi_a} + g^*\widehat{\phi_{ab^{-k}aba^{-1}b}}.
\]
Now clearly $\widehat{\phi_{ab^{-k}aba^{-1}b}} \to 0$ as $k \to \infty$ and hence also $g^*\widehat{\phi_{ab^{-k}aba^{-1}b}} \to g^*0 = 0$ by continuity of the action. We deduce that
\[
\widehat{\phi_{ab}} = \lim_{k \to \infty}(g_k^*\widehat{\phi_{a}} -g^*\widehat{\phi_a}) \in \overline{{\rm span}({\rm QOut}_{(2)}(F_n).\widehat{\phi_a}}).
\]
\end{proof}
Combining this with Grigorchuk's theorem we finally obtain the following result, which contains Theorem \ref{IntroMain} as a special case.
\begin{theorem}\label{ThmMain} For every $n \geq 1$ we have
\[\mathcal H(F_n) = \overline{{\rm span}({\rm QOut}_{(2)}(F_n).{\rm Hom}(F_n,\R))}.\]
\end{theorem}
\begin{proof} From Corollary \ref{Cor2orbits} and Corollary \ref{OrbitClosure} we deduce that
\[
\Phi(F_n,S) \subseteq  \overline{{\rm span}({\rm QOut}_{(2)}(F_n).{\rm Hom}(F_n,\R))}.
\]
However, by Theorem \ref{Grig} we have $\overline{{\rm span}({\Phi(F_n, S)})} = \mathcal H(F_n)$
\end{proof}
\section{Comparison to other definitions of quasimorphisms}\label{SecComparison}
The results in this article demonstrate that the definition of a quasimorphism between non-commutative groups provided by Definition \ref{DefQM} leads to a rich and substantial theory. If we take the (very classical) definition of a real-valued quasimorphism for granted, then it is the most general possible categorical definition of a quasimorphism. Here the word \emph{categorical} refers to the fact that we want the composition of two quasimorphisms to be again a quasimorphism. This seems to be a reasonable demand, and it is the only demand we make. One may criticize that our notion of quasimorphism is too general and try to define a more narrow notion of quasimorphism which is more closely modeled on the definition of a real-valued quasimorphism. In this section we discuss various such more restrictive notions of quasimorphisms and their properties.

One of the most classical sources concerning quasimorphisms is Chapter 6 of Ulam's book \cite{Ulam}. Among other things, Ulam defines a map $f: G \to H$ between a group $G$ and a metric group $(H, d_H)$ to be a $\delta$-homomorphism if
\[
\forall g_1, g_2 \in G:\quad d_H(f(g_1g_2), f(g_1)f(g_2))<\delta.
\]
If $H$ is the additive group of real numbers with the Euclidean distance, then this is precisely the definition of a quasimorphism (of defect at most $\delta$). However, Ulam's definition makes sense for arbitrary metric groups $H$. The best studied case besides the case of $\R$ is the one where $H$ is the unitary group of a (typically infinite-dimensional) Hilbert space, which leads to the study of Ulam stability. We refer the reader to \cite{BurgerOzawaThom} and the references therein for a recent account.

If $H$ is a discrete metric group (i.e., the topology of $d_H$ induced on $H$ is the discrete topology) then $f: G\to H$ is a $\delta$-homomorphism in the sense of Ulam for some $\delta$ if and only if there exists a finite subset $E \subseteq H$ such that 
\begin{equation}\label{Ulam}
\forall w_1,w_2 \in G: f(w_1w_2) \in f(w_1)f(w_2)E
\end{equation}
Let us call a map $f: G \to H$ satisfying this condition an \emph{Ulam quasimorphism}. Ulam quasimorphisms have recently been classified in the work of Fujiwara and Kapovich \cite{FujiwaraKapovich}. They are categorical in the sense defined above, thus form a subclass of quasimorphisms in the sense of Definition \ref{DefQM}. A slightly more general class of quasimorphisms is obtained by demanding only that there exists a finite subset $E \subseteq H$ such that 
\begin{equation}\label{weakUlam}
f(w_1 w_2) \in E f(w_1) E f(w_2) E.
\end{equation}
Such generalized Ulam quasimorphisms were introduced in \cite[Sec. 2.6]{FujiwaraKapovich} under the name of \emph{algebraic quasihomomorphisms} ``inspired by a correspondence from Narutaka Ozawa''. Let us call two quasmorphisms $f_1, f_2: G \to H$ \emph{algebraically equivalent} if there exists a finite subset $E \subseteq H$ such that
\[
f_2(g) \in Ef_1(H)E.
\]
Then algebraic quasihomomorphisms, unlike Ulam quasimorphisms, are closed under algebraic equivalence. It seems to be an open problem whether every algebraic quasihomomorphism is algebraically equivalent to an Ulam quasimorphism. Note that if $f: G \to H$ is an algebraic quasihomomorphism and $g \in G$, then 
\[
f(e) = f(gg^{-1}) = f(g^{-1}g) \in f(g)Ef(g^{-1}) \cap f(g^{-1})Ef(g)
\]
It follows that there exists a finite set $F$ such that $f(g^{-1}) \in Ff(g)^{-1}F$. Enlarging $E$ if necessary we may assume that in fact 
\begin{equation}\label{InverseCondition}
f(g^{-1}) \in Ef(g)^{-1}E.
\end{equation}
We are now going to present a generalization of algebraic quasihomomorphisms in the context of free groups which preserves this property, but otherwise demands \eqref{weakUlam} only for reduced words:
\begin{definition}
Let $G$ be a group, $F$ a free group and $f: F \to G$ a map. Then $f$ is called a \emph{quasi-Ulam quasimorphism} if there exists a finite set $E \subseteq G$ such that 
\begin{enumerate}
\item 
for all words $w_1$ and~$w_2$, for which~$w_1w_2$ is a reduced word, we have $f(w_1 w_2) \in E f(w_1) E f(w_2) E$ and
\item  $\forall g\in F\colon \, f(g^{-1}) \in E f(g)^{-1} E$.
\end{enumerate}
\end{definition}
By the previous remark every algebraic quasihomomorphism (and thus every Ulam quasimorphism) is quasi-Ulam. On the other hand, Proposition~\ref{thm:qmorphism:criterion:for:reduced:words} says precisely that every quasi-Ulam quasimorphism is indeed a quasimorphism in the sense of Definition \ref{DefQM}. Finally, it is easy to see that quasi-Ulam quasimorphism are closed under composition.

Requiring \eqref{weakUlam} only for reduced words looks like a minor technical modification of the definition at first sight. However, as we will show next, the consequences of this modification are quite dramatic. Let us denote by
\[
{\rm QOut}_{U}(F_n) <{\rm QOut}_{qU}(F_n) < {\rm QOut}(F_n)
\]
the subgroups given by all equivalence classes of bijective Ulam, respectively quasi-Ulam quasimorphisms. We observe that all examples of quasioutomorphisms of free groups constructed in the present article are quasi-Ulam. (The reader can check that we used Proposition \ref{thm:qmorphism:criterion:for:reduced:words} in each of the proofs.) We thus obtain the following generalization of Theorem \ref{ThmMain} with the same proof:
\begin{theorem}\label{qU}
For every $n \geq 1$ denote by ${\rm QOut}_{qU}(F_n) < {\rm QOut}(F_n)$ the subgroup generated by all equivalence classes of bijective quasi-Ulam quasimorphisms. Then
\[ \overline{{\rm span}({\rm QOut}_{qU}(F_n).{\rm Hom}(F_n,\R))} = \mathcal H(F_n).\]
\end{theorem}
For actual Ulam quasimorphism the picture is entirely different. It it is an immediate consequence of \cite[Theorem 3]{FujiwaraKapovich} that every bijective Ulam quasimorphism from a non-abelian free group $F$ to itself is at bounded distance from an automorphism. This implies:
\begin{theorem}[Fujiwara-Kapovich]
For every $n \geq 1$ denote by ${\rm QOut}_{U}(F_n) < {\rm QOut}(F_n)$ the subgroup generated by all equivalence classes of bijective Ulam quasimorphisms. Then
\[\overline{{\rm span}({\rm QOut}_{U}(F_n).{\rm Hom}(F_n,\R))} = {\rm Hom}(F_n, \R).\]
\end{theorem}
We do not know, whether the theorem remains true if we replace Ulam quasimorphisms by algebraic quasihomomorphisms. To prove this, it would suffice to show that every algebraic quasihomomorphism is weakly equivalent  to an Ulam quasimorphism, but as mentioned before this is an open problem. In any case, we see that the class of quasi-Ulam quasimorphisms is sufficiently general to provide many interesting examples of quasioutomorphisms on free groups (which the class of Ulam quasimorphisms is not) and at the same time still admits a concrete definition in the spirit of Ulam. It thus forms a very interesting class of quasimorphisms on free groups, which deserves further study.

For general hyperbolic groups it is not obvious what a natural class of algebraically defined quasimorphisms to study is. The results of Fujiwara and Kapovich show also in this case that the class of Ulam quasimorphisms is too restrictive. One could define quasi-Ulam quasimorphisms on finitely generated group $\Gamma $ by demanding that they are quasimorphisms which pull back to quasi-Ulam quasimorphisms for some (or a fixed or every) surjective homomorphism $F_n \to \Gamma$, but whether this yields an interesting theory remains to be seen. 
Generally speaking, it would be desirable to find more examples of  quasioutomorphisms of finitely generated groups. However, this goal is beyond the scope of the present article.

\bigskip

\textbf{Acknowledgments.} The core of this work was carried out during a fruitful stay of the authors at the Mittag-Leffler-Institut, Djursholm in Spring 2012. The authors would like to thank the organizers of the program on Geometric and Analytic Aspects of Group Theory
and the staff of the institute for the excellent working conditions. They are also indebted to Koji Fujiwara, Misha Kapovich, Anders Karlsson,  Denis Osin and Alexey Talambutsa for comments and remarks. Finally,  they would like to thank Uri Bader for suggesting Definition \ref{DefQM} which is at the heart of the present work.

\bibliography{qm}    

\begin{thebibliography}{10}

\bibitem{BargeGhys}
J.~Barge and {\'E}.~Ghys.
\newblock Surfaces et cohomologie born\'ee.
\newblock {\em Invent. Math.}, 92(3):509--526, 1988.

\bibitem{BSHLie}
G.~Ben~Simon and T.~Hartnick.
\newblock Invariant orders on {H}ermitian {L}ie groups.
\newblock {\em J. Lie Theory}, 22(2):437--463, 2012.

\bibitem{FujiwaraEtAl}
M.~Bestvina, K.~Bromberg, and K.~Fujiwara.
\newblock Bounded cohomology via quasi-trees.
\newblock Preprint.

\bibitem{BestvinaFujiwara}
M.~Bestvina and K.~Fujiwara.
\newblock Bounded cohomology of subgroups of mapping class groups.
\newblock {\em Geom. Topol.}, 6:69--89 (electronic), 2002.

\bibitem{Brooks}
R.~Brooks.
\newblock Some remarks on bounded cohomology.
\newblock In {\em Riemann surfaces and related topics: {P}roceedings of the
  1978 {S}tony {B}rook {C}onference ({S}tate {U}niv. {N}ew {Y}ork, {S}tony
  {B}rook, {N}.{Y}., 1978)}, volume~97 of {\em Ann. of Math. Stud.}, pages
  53--63. Princeton Univ. Press, Princeton, N.J., 1981.

\bibitem{BuMo1}
M.~Burger and N.~Monod.
\newblock Bounded cohomology of lattices in higher rank {L}ie groups.
\newblock {\em J.~Eur.~Math.~Soc.}, 1(2):199--235, 1999.

\bibitem{BuMo2}
M.~Burger and N.~Monod.
\newblock Continuous bounded cohomology and applications to rigidity theory.
\newblock {\em Geom. Funct. Anal.}, 12(2):219--280, 2002.

\bibitem{BurgerOzawaThom}
M.~Burger, N.~Ozawa, and A.~Thom.
\newblock On {U}lam stability.
\newblock {\em Israel J. Math.}, 193(1):109--129, 2013.

\bibitem{scl}
D.~Calegari.
\newblock {\em scl}, volume~20 of {\em MSJ Memoirs}.
\newblock Mathematical Society of Japan, Tokyo, 2009.

\bibitem{DGO}
F.~Dahmani, V.~Guirardel, and D.~Osin.
\newblock Hyperbolically embedded subgroups and rotating families in groups
  acting on hyperbolic spaces.
\newblock Preprint.

\bibitem{EpsteinFujiwara}
D.~B.~A. Epstein and K.~Fujiwara.
\newblock The second bounded cohomology of word-hyperbolic groups.
\newblock {\em Topology}, 36(6):1275--1289, 1997.

\bibitem{FujiwaraKapovich}
K.~Fujiwara and M.~Kapovich.
\newblock On quasihomomorphisms with noncommutative targets.
\newblock Preprint.

\bibitem{Grigorchuk}
R.~I. Grigorchuk.
\newblock Some results on bounded cohomology.
\newblock In {\em Combinatorial and geometric group theory ({E}dinburgh,
  1993)}, volume 204 of {\em London Math. Soc. Lecture Note Ser.}, pages
  111--163. Cambridge Univ. Press, Cambridge, 1995.

\bibitem{Gromov}
M.~Gromov.
\newblock Volume and bounded cohomology.
\newblock {\em Inst. Hautes \'Etudes Sci. Publ. Math.}, 56:5--99 (1983), 1982.

\bibitem{OsinHull}
M.~Hull and D.~Osin.
\newblock Induced quasi-cocycles on groups with hyperbolically embedded
  subgroups.
\newblock Preprint.

\bibitem{Ivanov}
N.~V. Ivanov.
\newblock Foundations of the theory of bounded cohomology.
\newblock {\em Zap. Nauchn. Sem. Leningrad. Otdel. Mat. Inst. Steklov. (LOMI)},
  143:69--109, 177--178, 1985.
\newblock Studies in topology, V.

\bibitem{Johnson}
B.~E. Johnson.
\newblock {\em Cohomology in {B}anach algebras}.
\newblock American Mathematical Society, Providence, R.I., 1972.
\newblock Memoirs of the American Mathematical Society, No. 127.

\bibitem{MR1655470}
G.~Levitt and J.-L. Nicolas.
\newblock On the maximum order of torsion elements in {${\rm GL}(n,{\bf Z})$}
  and {${\rm Aut}(F_n)$}.
\newblock {\em J. Algebra}, 208(2):630--642, 1998.

\bibitem{MonodDiss}
N.~Monod.
\newblock {\em Continuous bounded cohomology of locally compact groups}, volume
  1758 of {\em Lecture Notes in Mathematics}.
\newblock Springer-Verlag, Berlin, 2001.

\bibitem{RemyMonod}
N.~Monod and B.~R{\'e}my.
\newblock Boundedly generated groups with pseudocharacters.
\newblock {\em J. London Math. Soc. (2)}, 73(1):104--108, 2006.
\newblock [Appendix to the paper \emph{Quasi-actions on trees and property
  ({QFA})} by J. F. Manning].

\bibitem{Osin}
D.~Osin.
\newblock Acylindrically hyperbolic groups.
\newblock Preprint.

\bibitem{Rolli}
P.~Rolli.
\newblock Split quasicocycles.
\newblock Preprint.

\bibitem{Ulam}
S.~M. Ulam.
\newblock {\em A collection of mathematical problems}.
\newblock Interscience Tracts in Pure and Applied Mathematics, no. 8.
  Interscience Publishers, New York-London, 1960.

\end{thebibliography}
\end{document}